\documentclass[final,5p,times,10pt]{elsarticle}
\usepackage[centertags]{amsmath}
\usepackage{amsmath,amsfonts,amssymb,amsthm,epstopdf,newlfont,graphicx}
\usepackage{booktabs,bm,enumitem,setspace,empheq,cancel}
\usepackage{algorithm}
\usepackage[short,nocomma]{optidef}
\usepackage{algpseudocode}
\usepackage{caption,subcaption}
\usepackage{threeparttable,multirow}
\usepackage{amsmath,accents}
\usepackage[T1]{fontenc}
\usepackage[american]{babel}
\usepackage{natbib}
\biboptions{comma,square,sort&compress}
\usepackage{appendix}
\usepackage[colorlinks=true]{hyperref}

\newcommand\hcancel[2][black]{\setbox0=\hbox{$#2$}%
\rlap{\raisebox{.25\ht0}{\textcolor{#1}{\rule{0.7\wd0}{0.75pt}}}}#2} 
\newcommand\hcancelt[2][black]{\setbox0=\hbox{$#2$}%
\rlap{\raisebox{.25\ht0}{\textcolor{#1}{\hspace{0.3mm}\rule{0.7\wd0}{0.75pt}}}}#2} 
\newtheorem{thm}{Theorem}[section]

\newtheorem{rem}{Remark}[section]
\theoremstyle{definition}

\numberwithin{algorithm}{section}
\numberwithin{equation}{section}
\renewcommand{\theequation}{\thesection.\arabic{equation}}
\def\simgt{\,\hbox{\lower0.6ex\hbox{$>$}\llap{\raise0.3ex\hbox{$\sim$}}}\,}
\def\simlt{\,\hbox{\lower0.6ex\hbox{$<$}\llap{\raise0.3ex\hbox{$\sim$}}}\,}
\def\simgteq{\,\hbox{\lower0.6ex\hbox{$\ge$}\llap{\raise0.6ex\hbox{$\sim$}}}\,}
\def\simlteq{\,\hbox{\lower0.6ex\hbox{$\le$}\llap{\raise0.6ex\hbox{$\sim$}}}\,}
\def\applteq{\,\hbox{\lower0.6ex\hbox{$\le$}\llap{\raise0.8ex\hbox{$\approx$}}}\,}
\def\applt{\,\hbox{\lower0.6ex\hbox{$<$}\llap{\raise0.5ex\hbox{$\approx$}}}\,}
\DeclareMathAlphabet\mathbfcal{OMS}{cmsy}{b}{n}

\makeatletter
\def\user@resume{resume}
\def\user@intermezzo{intermezzo}
\newcounter{previousequation}
\newcounter{lastsubequation}
\newcounter{savedparentequation}
\makeatother
\newcommand{\C}[1]{\mathcal{#1}}

\newcommand{\F}[1]{\mathbf{#1}}

\newcommand{\bsC}[1]{\boldsymbol{\C{#1}}}

\newcommand{\MB}[1]{\mathbb{#1}}
\newcommand{\MBS}{\MB{S}}
\newcommand{\MBG}{\MB{G}}
\newcommand{\MBR}{\mathbb{R}}

\newcommand{\MBRP}{\MBR^+}
\newcommand{\MBRzer}{\MBR_0}
\newcommand{\MBRzerP}{\MBRzer^+}
\newcommand{\MBZ}{\mathbb{Z}}
\newcommand{\MBZP}{\MBZ^+}
\newcommand{\MBZzer}{\MBZ_0}
\newcommand{\MBZzerP}{\MBZzer^+}
\newcommand{\MBZe}{\MBZ_e}
\newcommand{\MBZeP}{\MBZe^+}
\newcommand{\MBZzereP}{\MBZ_{0,e}^+}
\newcommand{\MBZOP}{\MBZ_{o}^+}
\newcommand{\MBT}{\mathbb{T}}
\newcommand{\MBJ}{\mathbb{J}}
\newcommand{\MBN}{\mathbb{N}}
\newcommand{\MFF}{\mathfrak{F}}
\newcommand{\MFP}{\mathfrak{P}}
\newcommand{\MFC}{\mathfrak{C}}
\newcommand{\MFK}{\mathfrak{K}}
\newcommand{\MFX}{\mathfrak{X}}
\newcommand{\MFU}{\mathfrak{U}}

\newcommand{\canczer}[1]{#1_{\hcancel{0}}}
\newcommand{\cancbra}[1]{\hcancel{[}#1\hcancelt{]}}
\newcommand{\sumd}{\sideset{}{'}}
\newcommand{\bmc}{\bm{c}}

\newcommand{\bmx}{\bm{x}}

\newcommand{\bmt}{\bm{t}}

\newcommand{\bmz}{\bm{z}}

\newcommand{\txs}{\tilde{x}^{*}}
\newcommand{\tus}{\tilde{u}^{*}}
\newcommand{\tbmx}{\tilde{\bmx}}
\newcommand{\tbmu}{\tilde{\bmu}}

\newcommand{\tbmxs}{\tbmx^*}
\newcommand{\tbmus}{\tbmu^*}

\newcommand{\bmy}{\bm{y}}

\newcommand{\bmh}{\bm{h}}

\newcommand{\bmu}{\bm{u}}

\newcommand{\bmf}{\bm{f}}

\newcommand{\bmzer}{\bm{\mathit{0}}}
\newcommand{\bmone}{\bm{\mathit{1}}}

\newcommand{\bmX}{\bm{X}}

\newcommand{\hbmz}{\hat\bmz}

\newcommand{\hz}{\hat{z}}

\newcommand{\FOmega}{\F{\Omega}}

\newcommand{\foralla}{\,\forall_{\mkern-6mu a}\,}
\newcommand{\foralle}{\,\forall_{\mkern-6mu e}\,}
\newcommand{\foralls}{\,\forall_{\mkern-6mu s}\,}

\newcommand{\Def}[1]{\text{Def}\left(#1\right)}

\newcommand{\GLD}[4]{{}_{\;\;#1}^{GL}D_{#2}^{#3}{#4}}
\newcommand{\RLD}[4]{{}_{\;\;#1}^{RL}D_{#2}^{#3}{#4}}
\newcommand{\CD}[4]{{}_{#1}^CD_{#2}^{#3}{#4}}
\newcommand{\MRLD}[4]{{}_{\hspace{3.6mm}#1}^{MRL}D_{#2}^{#3}{#4}}
\newcommand{\MCD}[4]{{}_{\;\;#1}^{MC}D_{#2}^{#3}{#4}}
\newcommand{\MD}[4]{{}_{#1}^{M}D_{#2}^{#3}{#4}}
\newcommand{\ED}[4]{{}_{#1}^{E}D_{#2}^{#3}{#4}}
\usepackage{mathtools}
\mathtoolsset{showonlyrefs}
\makeatletter
\def\BState{\State\hskip-\ALG@thistlm}
\makeatother
\MakeRobust{\eqref}
\errorcontextlines\maxdimen

\makeatletter
    \newcommand*{\algrule}[1][\algorithmicindent]{\makebox[#1][l]{\hspace*{.5em}\thealgruleextra\vrule height \thealgruleheight depth \thealgruledepth}}%
\newcommand*{\thealgruleextra}{}
\newcommand*{\thealgruleheight}{.75\baselineskip}
\newcommand*{\thealgruledepth}{.25\baselineskip}

\newcount\ALG@printindent@tempcnta
\def\ALG@printindent{%
    \ifnum \theALG@nested>0
        \ifx\ALG@text\ALG@x@notext
        \else
            \unskip
            \addvspace{-1pt}
            \ALG@printindent@tempcnta=1
            \loop
                \algrule[\csname ALG@ind@\the\ALG@printindent@tempcnta\endcsname]%
                \advance \ALG@printindent@tempcnta 1
            \ifnum \ALG@printindent@tempcnta<\numexpr\theALG@nested+1\relax
            \repeat
        \fi
    \fi
    }%
\usepackage{etoolbox}
\patchcmd{\ALG@doentity}{\noindent\hskip\ALG@tlm}{\ALG@printindent}{}{\errmessage{failed to patch}}
\makeatother

\newbox\statebox
\newcommand{\myState}[1]{%
    \setbox\statebox=\vbox{#1}%
    \edef\thealgruleheight{\dimexpr \the\ht\statebox+1pt\relax}%
    \edef\thealgruledepth{\dimexpr \the\dp\statebox+1pt\relax}%
    \ifdim\thealgruleheight<.75\baselineskip
        \def\thealgruleheight{\dimexpr .75\baselineskip+1pt\relax}%
    \fi
    \ifdim\thealgruledepth<.25\baselineskip
        \def\thealgruledepth{\dimexpr .25\baselineskip+1pt\relax}%
    \fi
    \State #1%
    \def\thealgruleheight{\dimexpr .75\baselineskip+1pt\relax}%
    \def\thealgruledepth{\dimexpr .25\baselineskip+1pt\relax}%
}
\makeatletter
\newcommand{\oset}[3][0ex]{%
  \mathrel{\mathop{#3}\limits^{
    \vbox to#1{\kern-2\ex@
    \hbox{$\scriptstyle#2$}\vss}}}}
\makeatother
\begin{document}
\begin{frontmatter}
\title{Fourier-Gegenbauer Pseudospectral Method for Solving Periodic Fractional Optimal Control Problems}
\author[XMUM,Assiut]{Kareem T. Elgindy\corref{cor1}}
\ead{kareem.elgindy@(xmu.edu.my;gmail.com)}
\address[XMUM]{Mathematics Department, School of Mathematics and Physics, Xiamen University Malaysia, Sepang 43900, Malaysia}
\address[Assiut]{Mathematics Department, Faculty of Science, Assiut University, Assiut 71516, Egypt}
\cortext[cor1]{Corresponding author}

\begin{abstract}
This paper introduces a new accurate model for periodic fractional optimal control problems (PFOCPs) using Riemann-Liouville (RL) and Caputo fractional derivatives (FDs) with sliding fixed memory lengths. The paper also provides a novel numerical method for solving PFOCPs using Fourier and Gegenbauer pseudospectral methods. By employing Fourier collocation at equally spaced nodes and Fourier and Gegenbauer quadratures, the method transforms the PFOCP into a simple constrained nonlinear programming problem (NLP) that can be treated easily using standard NLP solvers. We propose a new transformation that largely simplifies the problem of calculating the periodic FDs of periodic functions to the problem of evaluating the integral of the first derivatives of their trigonometric Lagrange interpolating polynomials, which can be treated accurately and efficiently using Gegenbauer quadratures. We introduce the notion of the $\alpha$th-order fractional integration matrix with index $L$ based on Fourier and Gegenbauer pseudospectral approximations, which proves to be very effective in computing periodic FDs. We also provide a rigorous priori error analysis to predict the quality of the Fourier-Gegenbauer-based approximations to FDs. The numerical results of the benchmark PFOCP demonstrate the performance of the proposed pseudospectral method.
\end{abstract}
\end{frontmatter}

\section{Introduction}
\label{Int}
Optimal control (OC) theory remains an active area of research, with numerous applications in engineering, computer science, physics, astronomy, materials science, decision sciences, and many other fields. Most studies related to this area have focused on OC problems (OCPs) governed by integer order dynamical systems equations; however, the past decade has witnessed  increasing interest and rapid expansion of research works on fractional OCPs (FOCPs) as a viable alternative to classical OCPs, motivated by the fact that modeling various dynamic systems can be described more accurately using systems of fractional differential equations (FDEs); cf. \cite{agrawal2004general,agrawal2006formulation,agrawal2007hamiltonian,
li2008fractional,jelicic2009optimality,almeida2015discrete,
marzban2019solution,dehestani2022spectral}.

This paper presents a new formulation and numerical solution scheme for a class of periodic FOCPs (PFOCPs) governed by fractional dynamic systems. Fractional dynamic systems are systems whose dynamics are described by FDEs, and PFOCP is an OCP described by fractional dynamic systems with periodic solutions. Classical Riemann-Liouville (RL) and Caputo fractional derivatives (FDs) of non-constant $T$-periodic functions are not $T$-periodic, as shown earlier by \citet{tavazoei2010note}, limiting their use to model PFOCPs. This limitation triggers the crucial need to define new fractional differentiation operators that can preserve the periodicity of a periodic function and allow for the existence of periodic solutions to PFOCPs. 

One approach to defining such operators is to modify the classical RL and Caputo fractional differentiation formulas by fixing their memory length and varying their lower terminals, as recently introduced by \citet{bourafa2021periodic}, giving rise to a periodic fractional differentiation operator that alleviates the aforementioned limitations. We used this periodic fractional differentiation operator in this study to define a new class of PFOCPs suitable for controlling fractional dynamic systems. We demonstrate how to simplify the definition of the periodic fractional differentiation operator using a novel and clever change of variables that permits the evaluation of the FD using Fourier interpolation and Gegenbauer quadrature. For a fractional order $\alpha \in (0,1)$, the proposed transformation magically reduces the integral defining the FD to one that includes the first derivative of the trigonometric Lagrange interpolating polynomial associated with the Fourier interpolation of the periodic function. The aftermath of this procedure is the establishment of the Fourier-Gegenbauer-based pseudospectral (FGPS) fractional integration matrix (or simply the FGPSFIM) of Toeplitz structure that can deliver very accurate and efficient approximations to periodic FDs of periodic functions at any set of mesh points in their domains. Another contribution of this study is the derivation of the error bounds associated with the FGPS quadrature induced by the FGPSFIM, which highlights the quality of FGPS approximations to periodic FDs in light of some of the main parameters associated with the periodic FD and the proposed FGPS method. To the best of our knowledge, this paper introduces a new area of OC theory, which we call ``the periodic fractional OC theory.'' Moreover, this paper presents the FGPS method, which is a pioneering work to solve problems in this area numerically with exponential rates of convergence using FGPS approximations.

The remainder of this study is structured as follows. In Section \ref{sec:PN}, we provide preliminaries and notations to simplify the presentation of this paper. The PFOCP is introduced in general form in Section \ref{sec:PS1}. In Section \ref{sec:FGPMFD1}, we derive the FGPS formulas required to approximate the periodic FDs. Section \ref{sec:ECA1} presents error and convergence analyses of the derived FGPS formulas. The proposed FGPS method is introduced in Section \ref{sec:FGPM1}. We show the efficiency and accuracy of the proposed method in Section \ref{sec:NS1} followed by concluding remarks in Section \ref{sec:Conc} and some future works in Section \ref{sec:FW1}.

\section{Preliminaries and Notations}
\label{sec:PN}
The following notations are used throughout this study to abridge and simplify the mathematical formulas. Many of these notations appeared earlier in \cite{Elgindy2023a}; however, for convenience and to keep the paper self-explanatory, we summarize them below together with the new notations.

\noindent\textbf{Logical Symbols.} The  symbols $\forall, \foralla, \foralle$, and $\foralls$ stand for the phrases ``for all,'' ``for any,'' ``for each,'' and ``for some,'' in respective order.\\[0.5em]
\textbf{List and Set Notations.} $\MFC, \MFF$, and $\MFP$ denote the set of all complex-valued, real-valued, and piecewise continuous real-valued functions, respectively. Moreover, $\MBR, \MBRP, \canczer{\MBR}$, $\MBZ, \MBZP, \MBZzerP$, $\MBZOP$, $\MBZeP$, and $\MBZzereP$ denote the sets of real numbers, positive real numbers, nonzero real numbers, integers, positive integers, non-negative integers, positive odd integers, positive even integers, and non-negative even integers, respectively. When we overset any of the above sets by a right arrow we mean the subset of that set containing sufficiently large numbers; for example, $\oset{\rightarrow}{\MBZP}$ stands for the set of all sufficiently large positive integers. The notations $i$:$j$:$k$ or $i(j)k$ indicate a list of numbers from $i$ to $k$ with increment $j$ between numbers, unless the increment equals one where we use the simplified notation $i$:$k$. For example, $0$:$0.5$:$2$ simply means the list of numbers $0, 0.5, 1, 1.5$, and $2$, while $0$:$2$ means $0, 1$, and $2$. The list of symbols $y_1, y_2, \ldots, y_n$ is denoted by $\left. y_i \right|_{i=1:n}$ or simply $y_{1:n}$, and their set is represented by $\{y_{1:n}\}\,\foralla n \in \MBZP$. We define $\MBJ_n = \{0:n-1\}, \MBN_n = \{1:n\}\,\foralla n \in \MBZP$, and $\MFK_N = \{-N/2:N/2\}\,\foralla N \in \MBZeP$. $\MBS_n^{T} = \left\{t_{n,0:n-1}\right\}$ is the set of $n$ equally-spaced points such that $t_{n,j} = T j/n\, \forall j \in \MBJ_n$. $\MBG_n^{\lambda} = \left\{z_{n,0:n}^{\lambda}\right\}$ is the set of Gegenbauer-Gauss (GG) zeros of the $(n+1)$st-degree Gegenbauer polynomial with index $\lambda > -1/2$; cf. \cite{Elgindy201382}. Finally, the specific interval $[0, T]$ is denoted by $\FOmega_{T}\,\forall T > 0$; for example, $[0, t_{n,j}]$ is denoted by ${\FOmega_{t_{n,j}}}\,\forall j \in \MBJ_n$.\\[0.5em] 
\textbf{Function Notations.} $\delta_{n,m}$ is the usual Kronecker delta function of variables $n$ and $m$. $\left\lceil  \cdot  \right\rceil, \left\lfloor {\cdot} \right\rfloor$ and $\Gamma$ denote the ceil, floor, and Gamma functions, respectively. $\digamma$ denotes the digamma function defined as the logarithmic derivative of the gamma function. $\left( {\begin{array}{*{20}{c}}
\alpha \\
k
\end{array}} \right)$ is the binomial coefficient indexed by the pair $\alpha \in \MBR$ and $k \in \MBZzerP$. $\Im(\cdot)$ is the imaginary part of a complex number. $(\alpha)^{(k)} = \prod\limits_{l = 0}^{k-1} {(\alpha  - l)}$ is the $k$th-factorial power (the falling factorial) function of $\alpha \in \MBR$. For convenience, we shall denote $g(t_{n})$ by $g_n \foralla g \in \MFC, n \in \MBZ, t_n \in \MBR$, unless stated otherwise.\\[0.5em]
\textbf{Integral Notations.} We denote $\int_0^{b} {h(t)\,dt}$ and $\int_a^{b} {h(t)\,dt}$ by $\C{I}_{b}^{(t)}h$ and $\C{I}_{a, b}^{(t)}h$, respectively, $\foralla$ integrable $h \in \MFC, \{a, b\} \subset \MBR$. If the integrand function $h$ is to be evaluated at any other expression of $t$, say $u(t)$, we express $\int_0^{b} {h(u(t))\,dt}$ and $\int_a^b {h(u(t))\,dt}$ with a stroke through the square brackets as $\C{I}_{b}^{(t)}h\cancbra{u(t)}$ and $\C{I}_{a,b}^{(t)}h\cancbra{u(t)}$ in respective order.\\[0.5em] 
\textbf{Space and Norm Notations.} $\MBT_{T}$ is the space of $T$-periodic, univariate functions $\foralla T \in \MBRP$. $\Def{\FOmega}$ is the space of all functions defined on the set $\FOmega$. $C^k(\FOmega)$ is the space of $k$ times continuously differentiable functions on ${\FOmega}\,\forall k \in \MBZzerP$. ${}_T\mathfrak{X}_{n_1}^{n_2} = \{[y_0, \ldots, y_{n_1}]^{\top}: \MBRzerP \to \MBR^{n_1}\text{ s.t. }y_j \in \MBT_{T} \cap C^{n_2}(\MBRzerP) \forall j \in \MBN_{n_1}\}$ is the space of $T$-periodic, $n_2$-times continuously differentiable, $n_1$-dimensional vector functions on $\MBRzerP$. ${}_T\MFU_{n} = \{[u_0, \ldots, u_{n}]^{\top}: \MBRzerP \to \MBR^{n}\text{ s.t. }u_j \in \MBT_{T} \cap \MFP\,\forall j$ $\in \MBN_{n}\}$ is the space of $T$-periodic, $n$-dimensional piecewise continuous vector functions on $\MBRzerP$. $L^p({\FOmega})$ is the Banach space of measurable functions $u$ defined on ${\FOmega}$ such that ${\left\| u \right\|_{{L^p}}} = {\left( {{\C{I}_{\FOmega}}{{\left| u \right|}^p}} \right)^{1/p}}$ $< \infty\,\forall p \ge 1$. In particular, we write $\left\|u\right\|_{\infty}$ to denote ${\left\| u \right\|_{{L^{\infty}}}}$. Finally, $\left\|\cdot\right\|_2$ denotes the usual Euclidean norm of vectors.\\[0.5em]
\textbf{Vector Notations.} We shall use the shorthand notations $\bmt_N$ and $g_{0:N-1}$ to stand for the column vectors $[t_{0}, t_{1}, \ldots$, $t_{N-1}]^{\top}$ and $[g_0, g_1, \ldots, g_{N-1}]^{\top}\,\forall N \in \MBZP$ in respective order. In general, $\foralla h \in \MFC$ and a vector $\bmy$ whose $i$th-element is $y_i \in \MBR$, the notation $h(\bmy)$ stands for a vector of the same size and structure of $\bmy$ such that $h(y_i)$ is the $i$th element of $h(\bmy)$. Moreover, by $\bmh(\bmy)$ or $h_{1:m}\cancbra{\bmy}$ with a stroke through the square brackets, we mean $[h_1(\bmy), \ldots, h_m(\bmy)]^{\top}\,\foralla m$-dimensional column vector function $\bmh$, with the realization that the definition of each array $h_i(\bmy)$ follows the former notation rule $\foralle i$. Furthermore, if $\bmy$ is a vector function, say $\bmy = \bmy(t)$, then we write $h(\bmy(\bmt_N))$ and $\bmh(\bmy(\bmt_N))$ to denote $[h(\bmy(t_0)), h(\bmy(t_1)), \ldots, h(\bmy(t_{N-1}))]^{\top}$ and $[\bmh(\bmy(t_0)), \bmh(\bmy(t_1)), \ldots, \bmh(\bmy(t_{N-1}))]^{\top}$ in respective order.\\[0.5em] 
\textbf{Matrix Notations.} $\F{O}_n, \F{1}_n$, and $\F{I}_n$ stand for the zero, all ones, and the identity matrices of size $n$. $\F{C}_{n,m}$ indicates that $\F{C}$ is a rectangular matrix of size $n \times m$; moreover, $\F{C}_n$ denotes a row vector whose elements are the $n$th-row elements of $\F{C}$, except when $\F{C}_n = \F{O}_n, \F{1}_n$, or $\F{I}_n$, where it denotes the size of the matrix. For convenience, a vector is represented in print by a bold italicized symbol while a two-dimensional matrix is represented by a bold symbol, except for a row vector whose elements form a certain row of a matrix where we represent it in bold symbol as stated earlier. For example, $\bmone_n$ and $\bmzer_n$ denote the $n$-dimensional all ones- and zeros- column vectors, while $\F{1}_n$ and $\F{O}_n$ denote the all ones- and zeros- matrices of size $n$, respectively. Finally, the notation $[.;.]$ denotes the usual vertical concatenation.\\[0.5em]
\textbf{Special Mathematical Constants.} $\gamma_{em}$ is the Euler-Mascheroni constant defined by 
\[\gamma_{em} = \mathop {\lim }\limits_{n \to \infty } \left( { -\ln n + \sum\limits_{k = 1}^n {\frac{1}{k}} } \right) \approx 0.577216,\]
rounded to six decimal digits. $\epsilon_{\text{mach}}$ is the machine epsilon, which gives the upper bound on the relative approximation error due to rounding in floating point arithmetic.\\[0.5em]
\textbf{Common Fractional Differentiation Formulas.} Let $\alpha \in \MBRP$, $m = \left\lceil  \alpha  \right\rceil$, and $f \in \Def{\FOmega_T}$. The $\alpha$-th order Gr\"{u}nwald-Letnikov derivative of $f$ with respect to $t$ and a terminal value $a$ is given by
\begin{equation}
\GLD{a}{t}{\alpha}{f(t)} = \mathop {\lim }\limits_{\scriptstyle h \to 0\atop
\scriptstyle nh = t - a} {h^{ - \alpha }}\sum\limits_{k = 0}^n {{{( - 1)}^k}\left( {\begin{array}{*{20}{c}}
\alpha \\
k
\end{array}} \right)f(t - kh)}.
\end{equation}
The $\alpha$-th order left RL and Caputo FDs are denoted by $\RLD{0}{t}{\alpha}{f(t)}$ and $\CD{0}{t}{\alpha}{f(t)}$, respectively, and are defined for $t \in \FOmega_T$ by
\begin{align}
\RLD{0}{t}{\alpha}{f(t)} &= \left\{ \begin{array}{l}
\frac{1}{{\Gamma (m - \alpha )}}\frac{{{d^m}}}{{d{t^m}}}\C{I}_{t}^{(\tau)}\left[{{{(t - \tau )}^{m - \alpha  - 1}}f}\right],\quad \alpha \notin \MBZP,\\
f^{(m)}(t),\quad \alpha \in \MBZP,
\end{array} \right.\\
\CD{0}{t}{\alpha}{f(t)} &= \left\{ \begin{array}{l}
\frac{1}{{\Gamma (m - \alpha )}} \C{I}_t^{(\tau)} \left[{{{(t - \tau )}^{m - \alpha  - 1}}f^{(m)}}\right],\quad \alpha \notin \MBZP,\\
f^{(m)}(t),\quad \alpha \in \MBZP.
\end{array} \right.
\end{align}
None of these fractional-order derivatives however is a $T$-periodic function when $f$ is a $T$-periodic function as shown by the following theorem.
\begin{thm}[\cite{tavazoei2010note}]\label{thm:1}
Suppose that $f$ is a nonconstant $T$-periodic function. If $f^{(0:m)}$ exist, then $\GLD{0}{t}{\alpha}{f(t)}, \RLD{0}{t}{\alpha}{f(t)}$, and $\CD{0}{t}{\alpha}{f(t)}$ cannot be $T$-periodic functions $\foralla \alpha \notin \MBZP$. 
\end{thm}
The aftermath of this theorem is that autonomous fractional-order systems whose differential equations contain only a fractional-order derivative defined based on the Gr\"{u}nwald-Letnikov, RL, or Caputo definitions cannot have nonconstant periodic solutions \cite{tavazoei2009proof,tavazoei2010note,kang2015nonexistence,kaslik2012non}. One approach to preserve the periodicity of RL and Caputo fractional-order operators and allow for the existence of periodic solutions of fractional-order models is to fix their memory length and vary their lower terminals, as shown in \cite{bourafa2021periodic}. In particular, the modified fractional operators, referred to by the RL and Caputo FDs with sliding fixed memory length $L > 0$, and denoted by $\MRLD{L}{t}{\alpha}{f(t)}$ and $\MCD{L}{t}{\alpha}{f(t)}$, respectively, are defined by
\begin{align}
\MRLD{L}{t}{\alpha}{f(t)} = \left\{ \begin{array}{l}
\frac{1}{{\Gamma (m - \alpha )}}\frac{{{d^m}}}{{d{t^m}}}\C{I}_{t-L, t}^{(\tau)}\left[{{{(t - \tau )}^{m - \alpha  - 1}}f}\right],\quad \alpha \notin \MBZP,\\
f^{(m)}(t),\quad \alpha \in \MBZP,
\end{array} \right.\\
\MCD{L}{t}{\alpha}{f(t)} = \left\{ \begin{array}{l}
\frac{1}{{\Gamma (m - \alpha )}} \C{I}_{t-L, t}^{(\tau)}\left[{{{(t - \tau )}^{m - \alpha  - 1}}f^{(m)}}\right],\quad \alpha \notin \MBZP,\label{eq:PMFCD1}\\
f^{(m)}(t),\quad \alpha \in \MBZP.
\end{array} \right.
\end{align}
If $f \in C^{(m)}(\MBRzerP)$, then $\MRLD{L}{t}{\alpha}{f(t)} = \MCD{L}{t}{\alpha}{f(t)}$, so we can denote both modified fractional operators by $\MD{L}{t}{\alpha}{}$. 

\begin{rem}
The reader may consult \cite[Theorem 3.9]{bourafa2021periodic} for a proof that the modified derivative $\MD{L}{t}{\alpha}{f}$ indeed preserves the periodicity of $f \foralla$ periodic function $f \in C^{m}[a-L,b]: \alpha \in \MBRP\backslash\MBZP, L \in \MBRP, m = \left\lceil  \alpha  \right\rceil, \{a,b\} \subset \MBR: a < b$. The reader may also consult Sections 3.4 and 3.5 of the same reference for a comparison between classical fractional-order derivatives and the modified derivative, in addition to two examples, including one on a physical model, to support the consistency of the modified derivative motif.
\end{rem}

\section{Problem Statement}
\label{sec:PS1}
In this section, we consider the general form of PFOCPs governed by $\MD{L}{t}{\alpha}{}$-based FDEs of order $0 < \alpha < 1$. In particular, our aim is to approximate the optimal $T$-periodic solutions of the following PFOCP:
\begin{mini}
  {\bmu}{J(\bmu) = \frac{1}{T} \C{I}_{T}^{(t)} {g}\cancbra{\bmx(t), \bmu(t), t}}{}{}
  {\label{eq:OC1}}{}
  \addConstraint{\MD{L}{t}{\alpha}{\bmx(t)}}{= \bmf\left(\bmx(t),\bmu(t),t\right)}{\quad \forall t \in {\FOmega_T},}
  \addConstraint{\bmc(\bmx(t)}{, \bmu(t), t) \le \bmzer_p}{\quad \forall t \in {\FOmega_T},}
\end{mini}
where $p \in \MBZP, \bmx \in {}_T\MFX_{n_x}^1, \bmu \in {}_T\MFU_{n_u}, g: \MBR^{n_x} \times \MBR^{n_u} \times \MBRzerP \to \MBR, \bmf = [f_1, \ldots, f_n]^{\top}: \MBR^{n_x} \times \MBR^{n_u} \times \MBRzerP \to \MBR^{n_x}$, and $\bmc = [c_1, \ldots, c_p]^{\top}: \MBR^{n_x} \times \MBR^{n_u} \times \MBRzerP \to \MBR^p$ with $\{g, f_i, c_i\} \subset C^{k}(\MBRzerP)\,\foralls k \in \MBZzerP$. Here, $\bmx$ and $\bmu$ are the state and control vector functions, respectively. We assume that a solution of the problem exists.

\section{The FGPS Approximation of \texorpdfstring{$\MD{L}{t}{\alpha}{}$}{The FGPS Approximation of the RL and Caputo FDs With Sliding Fixed Memory Length}}
\label{sec:FGPMFD1}
The change of variables 
\begin{equation}\label{eq:Elg1}
\tau = t-L\,y^{\frac{1}{1-\alpha}},
\end{equation}
transforms $\MD{L}{t}{\alpha}{f(t)}$ defined by Eq. \eqref{eq:PMFCD1} into a reduced form, denoted by $\ED{L}{t}{\alpha}{f(t)}$, and defined by
\begin{equation}\label{eq:RedElgPMFCD1}
\ED{L}{t}{\alpha}{f(t)} = \frac{L^{m-\alpha}}{(1-\alpha) \Gamma(m-\alpha)} \C{I}_1^{(y)} {\left[y^{\frac{m-1}{1-\alpha}} f^{(m)}\cancbra{t-L\,y^{\frac{1}{1-\alpha}}}\right]},
\end{equation}
$\forall t \in \FOmega_T, \alpha \in \MBRP\backslash\MBZP$. For $0 < \alpha < 1$, Eq. \eqref{eq:RedElgPMFCD1} simplifies further into
\begin{equation}\label{eq:RedElgPMFCD101}
\ED{L}{t}{\alpha}{f(t)} = \frac{L^{1-\alpha}}{\Gamma(2-\alpha)} \C{I}_1^{(y)} {f'\cancbra{t-L\,y^{\frac{1}{1-\alpha}}}}\quad \forall t \in \FOmega_T.
\end{equation}
In the sequel, we restrict our study of FGPS approximations for $\alpha \in (0, 1)$. To approximate Eq. \eqref{eq:RedElgPMFCD101}, consider the $N/2$-degree, $T$-periodic Fourier interpolant that matches a function $f \in \MBT_T$ at the set of nodes $\MBS_N^T$, denoted by ${I_N}f$, and defined by\begin{equation}\label{eq:eqLF1}
{I_N}f(t) = \sum\limits_{j = 0}^{N - 1} {{f_j}{\C{F}_j}(t)},
\end{equation}
where ${\C{F}_j}(t)$ is the $N/2$-degree trigonometric Lagrange interpolating polynomial given by
\begin{gather}
\C{F}_j(t) = \frac{1}{N}\sumd\sum\limits_{\left| k \right| \le N/2} {e^{i {\omega _k}(t - {t_{N,j}})}}\label{eq:CFL1}\\
 = \left\{ \begin{array}{l}
1,\quad t = t_{N,j},\nonumber\\
\frac{1}{N}\sin(N \nu_j)\cot({\nu_j}),\quad t \ne t_{N,j},
\end{array} \right.\quad \forall j \in \MBJ_N,
\end{gather}
${\omega _{a}} = 2\pi a/T\,\forall a \in \MBR, \nu_j = \pi \left( {t - {t_{N,j}}} \right)/T$, and the primed sigma denotes a summation in which the last term is omitted \cite{elgindy2019high}. For $f \in \MFF, \Im\left(\C{F}_j\right) = 0\,\forall j \in \MBJ$, and Eq. \eqref{eq:CFL1} simplifies into
\begin{equation}\label{eq:CFL2}
\C{F}_j(t) = \frac{1}{N}\sumd\sum\limits_{\left| k \right| \le N/2} {\cos\left({\omega _k}(t - {t_{N,j}})\right)}.
\end{equation}
Substituting Eq. \eqref{eq:eqLF1} into Eq. \eqref{eq:RedElgPMFCD101} yields the following approximation:
\begin{equation}\label{eq:RedElgPMFCD101app1}
\ED{L}{t}{\alpha}{f(t)} \approx \ED{L}{t}{\alpha}{I_Nf(t)} = \frac{L^{1-\alpha}}{\Gamma(2-\alpha)} \sum\limits_{j = 0}^{N - 1} {{f_j} \C{I}_1^{(y)}{{\C{F}'_j}\cancbra{t-L\,y^{\frac{1}{1-\alpha}}}}},
\end{equation}
where
\begin{align}
{\C{F}'_j}(t) = \left\{ \begin{array}{l}
\frac{2 \pi i}{N T}\sumd\sum\limits_{\scriptstyle\left| k \right| \le N/2\atop
\scriptstyle k \ne 0} {k{e^{i{\omega _k}(t - {t_{N,j}})}}}\quad \forall f \in \MFC,\\
-\frac{2 \pi}{N T}\sumd\sum\limits_{\scriptstyle\left| k \right| \le N/2\atop
\scriptstyle k \ne 0} {k\sin \left( {{\omega _k}(t - {t_{N,j}})} \right)} \quad \forall f \in \MFF.
\end{array} \right.
\end{align}

The integral $\C{I}_1^{(y)}{{\C{F}'_j}\cancbra{t-L\,y^{\frac{1}{1-\alpha}}}}$ is independent of $f$ and can be computed offline rapidly using the highly accurate shifted Gegenbauer (SG) quadratures; cf. \cite{Elgindy20161,Elgindy20171,elgindy2018optimal,Elgindy2023a}. In particular, we can approximate $\C{I}_1^{(y)}{{\C{F}'_j}\cancbra{t-L\,y^{\frac{1}{1-\alpha}}}}$ at any mesh point $t_{N,l} \in \MBS_{N}^T$ by the following formula \cite[Eq. (4.36)]{Elgindy20161}:
\begin{gather}
\C{I}_1^{(y)}{{\C{F}'_j}\cancbra{t_{N,l}-L\,y^{\frac{1}{1-\alpha}}}} \approx {}_{L}^E\C{Q}^{\alpha}_{N_G,l,j}\nonumber\\
= \frac{1}{2} \left[\F{P}\,{\C{F}'_j}\left(t_{N,l}\,\bmone_{N_G+1}-L\,\left(\hbmz_{N_G+1}^{\lambda}\right)^{\frac{1}{1-\alpha}}\right)\right],\label{eq:Quaderr1}
\end{gather}
where $\hz_{N_G,0:N_G}^{\lambda} = \frac{1}{2} \left(z_{N_G,0:N_G}^{\lambda}+1\right)\,\foralls N_G \in \MBZP$ are the SG-Gauss (SGG) points on the interval $\FOmega_1$ and $\F{P}$ is the $(N_G+1)$ dimensional, GG points-based integration row vector constructed using \cite[Algorithm 6 or 7]{Elgindy20171}. For relatively small quadrature nodes, we recommend to use instead the optimal Gegenbauer quadratures \cite{Elgindy201382}, which can provide more accurate integral approximations. Substituting Eq. \eqref{eq:Quaderr1} into \eqref{eq:RedElgPMFCD101app1} yields the following main approximation formula for the FD of the $T$-periodic function $f$ at any mesh point:
\begin{equation}\label{eq:RedElgPMFCD101app2}
\ED{L}{t_{N,l}}{\alpha}{f(t)} \approx \frac{L^{1-\alpha}}{\Gamma(2-\alpha)} \sum\limits_{j = 0}^{N - 1} {{}_{L}^E\C{Q}^{\alpha}_{N_G,l,j} {f_j}}\quad \forall l \in \MBJ_{N},
\end{equation}
or in matrix notation,
\begin{equation}\label{eq:RedElgPMFCD101app3}
\ED{L}{\bmt_{N}}{\alpha}{f(t)} \approx \frac{L^{1-\alpha}}{\Gamma(2-\alpha)} \left({}_{L}^E\bsC{Q}_{N_G}^{\alpha}\,f_{0:N-1}\right),
\end{equation}  
where $\ED{L}{\bmt_{N}}{\alpha}{f(t)} = \left[\ED{L}{t_{N,0}}{\alpha}{f(t)}, \ldots, \ED{L}{t_{N,N-1}}{\alpha}{f(t)}\right]^{\top}$, and
\begin{align*}
&{}_{L}^E\bsC{Q}^{\alpha}_{N_G} = \left[{}_{L}^E\bsC{Q}^{\alpha}_{N_G,0}; \ldots; {}_{L}^E\bsC{Q}^{\alpha}_{N_G,N-1}\right]:\\
&{}_{L}^E\bsC{Q}^{\alpha}_{N_G,l} = \left[{}_{L}^E\C{Q}^{\alpha}_{N_G,l,0}, \ldots, {}_{L}^E\C{Q}^{\alpha}_{N_G,l,N-1}\right]\quad \forall l \in \MBJ_N.
\end{align*}
We refer to Quadrature \eqref{eq:Quaderr1} and ${}_{L}^E\bsC{Q}^{\alpha}_{N_G}$ by the $\alpha$th-order Fourier-Gegenbauer-based pseudospectral quadrature (FGPSQ) with index $L$ and its associated integration matrix (FGPSFIM), respectively. Clearly, the FGPSFIM ${}_{L}^E\bsC{Q}^{\alpha}_{N_G}$ is a square, dense matrix of size $N$. In addition, the following theorem shows that ${}_{L}^E\bsC{Q}^{\alpha}_{N_G}$ is a Toeplitz (diagonal-constant) matrix in which each descending diagonal from left to right is constant.  

\begin{thm}\label{thm:Toeplitz1}
The FGPSFIM ${}_{L}^E\bsC{Q}^{\alpha}_{N_G}$ is a Toeplitz matrix.
\end{thm}
\begin{proof}
By Formula \eqref{eq:Quaderr1},
\begin{gather*}
{}_{L}^E\C{Q}^{\alpha}_{N_G,l+1,j+1} = \frac{1}{2} \left[\F{P}\,{\C{F}'_{j+1}}\left(t_{N,l+1}\,\bmone_{N_G+1}-L\,\left(\hbmz_{N_G+1}^{\lambda}\right)^{\frac{1}{1-\alpha}}\right)\right],\\
= \frac{\pi i}{N T} \F{P} \sumd\sum\limits_{\scriptstyle\left| k \right| \le N/2\atop
\scriptstyle k \ne 0} {k\,{\exp\left[i{\omega _k}\left(t_{N,l+1}\,\bmone_{N_G+1}-L\,\left(\hbmz_{N_G+1}^{\lambda}\right)^{\frac{1}{1-\alpha}} - {t_{N,j+1}}\,\bmone_{N_G+1}\right)\right]}}\\
= \frac{\pi i}{N T} \F{P} \sumd\sum\limits_{\scriptstyle\left| k \right| \le N/2\atop
\scriptstyle k \ne 0} {k\,{\exp\left[i{\omega _k}\left(t_{N,l}\,\bmone_{N_G+1}-L\,\left(\hbmz_{N_G+1}^{\lambda}\right)^{\frac{1}{1-\alpha}} - {t_{N,j}}\,\bmone_{N_G+1}\right)\right]}}\\
= \frac{1}{2} \left[\F{P}\,{\C{F}'_{j}}\left(t_{N,l}\,\bmone_{N_G+1}-L\,\left(\hbmz_{N_G+1}^{\lambda}\right)^{\frac{1}{1-\alpha}}\right)\right] = {}_{L}^E\C{Q}^{\alpha}_{N_G,l,j}\quad \forall  \{l,j\} \subset \MBJ_{N-1}.
\end{gather*}
It then follows that the FGPSFIM ${}_{L}^E\bsC{Q}^{\alpha}_{N_G}$ is a Toeplitz matrix and the proof is complete.
\end{proof}

Theorem \ref{thm:Toeplitz1} shows that the FGPSFIM ${}_{L}^E\bsC{Q}^{\alpha}_{N_G}$ requires only $2 N - 1$ storage. This is fortunate, because special algorithms such as the Levinson–Durbin algorithm, which exploits the Toeplitz structure, are much faster for solving Toeplitz systems and require only $O(N^2)$ flops and $O(N)$ storage \cite{bjorck2008numerical}. For information about the inversion and properties of Toeplitz matrices, the interested reader may consult \cite{LV20071189,Dan2016}.

The excellent approximations of the derived FGPS formulas \eqref{eq:RedElgPMFCD101app2} and \eqref{eq:RedElgPMFCD101app3} can be observed in Figures \ref{fig:0}--\ref{fig:2}, which show both the exact and FGPS approximate values of $\ED{L}{t}{\alpha}{\sin(t)}$ for the range $\alpha =$ 0.1:0.2:0.9, 0.99 using the parameter values $N \in \{4, 20, 40\}, L = 30,  N_G = 1000$, and $\lambda = 0$. Notice that the exact value of $\MD{L}{t}{\alpha}{\sin(t)}$ is given by
\begin{equation}
\MD{L}{t}{\alpha}{\sin(t)} = a \sin(t-L) + b \cos(t-L),
\end{equation}
where 
\begin{align}
a = L^{-\alpha} \left[E_{2,1-\alpha}\left(-L^2\right)-\sum_{k=0}^{\left\lfloor {\frac{m-1}{2}} \right\rfloor} {\frac{(-1)^k L^{2k}}{\Gamma(2k+1-\alpha)}}\right],\\
b = L^{1-\alpha} \left[E_{2,2-\alpha}\left(-L^2\right)-\sum_{k=0}^{\left\lfloor {\frac{m-2}{2}} \right\rfloor} {\frac{(-1)^k L^{2k}}{\Gamma(2k+2-\alpha)}}\right],
\end{align}
and $E_{a,b}(t)$ is the generalized Mittag-Leffler function defined by
\begin{equation}
E_{a,b}(t) = \sum_{k \in \MBZzerP} {\frac{t^k}{\Gamma(ak+b)}}\quad \forall a, b \in \MBRP;
\end{equation}
cf. \cite{bourafa2021periodic}. Notice also that $\ED{L}{t}{\alpha}{\sin(t)}$ is a $2\pi$-periodic function as clearly observed from Figure \ref{fig:2}, and the graph of $\ED{L}{t}{\alpha}{\sin(t)}$ converges to the graph of the cosine function as $\alpha \to 1$. 

\begin{figure}[t]
\centering
\includegraphics[scale=0.5]{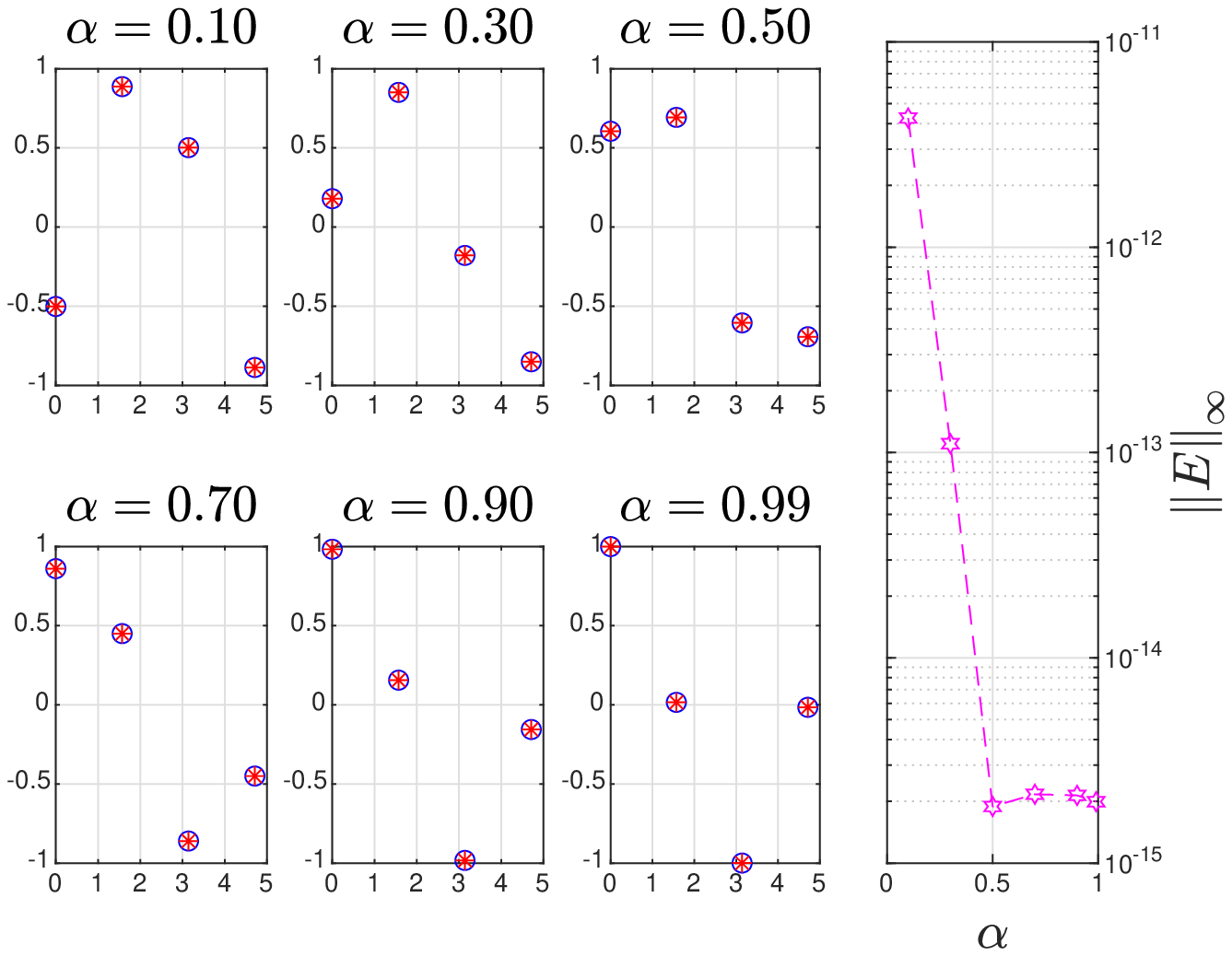}
\caption{The exact and FGPS approximate values of $\ED{L}{t}{\alpha}{\sin(t)}$ (Columns 1--3) and their corresponding maximum absolute errors (the 4th Column) for $\alpha = 0.1:0.2:0.9, 0.99$. The FGPS approximations were obtained using $N = 4, L = 30,  N_G = 1000$, and $\lambda = 0$. The exact values are shown in red color with * marker symbol, while the FGPS approximations are shown in blue colors with o marker symbol.}
\label{fig:0}
\end{figure}

\begin{figure}[t]
\centering
\includegraphics[scale=0.5]{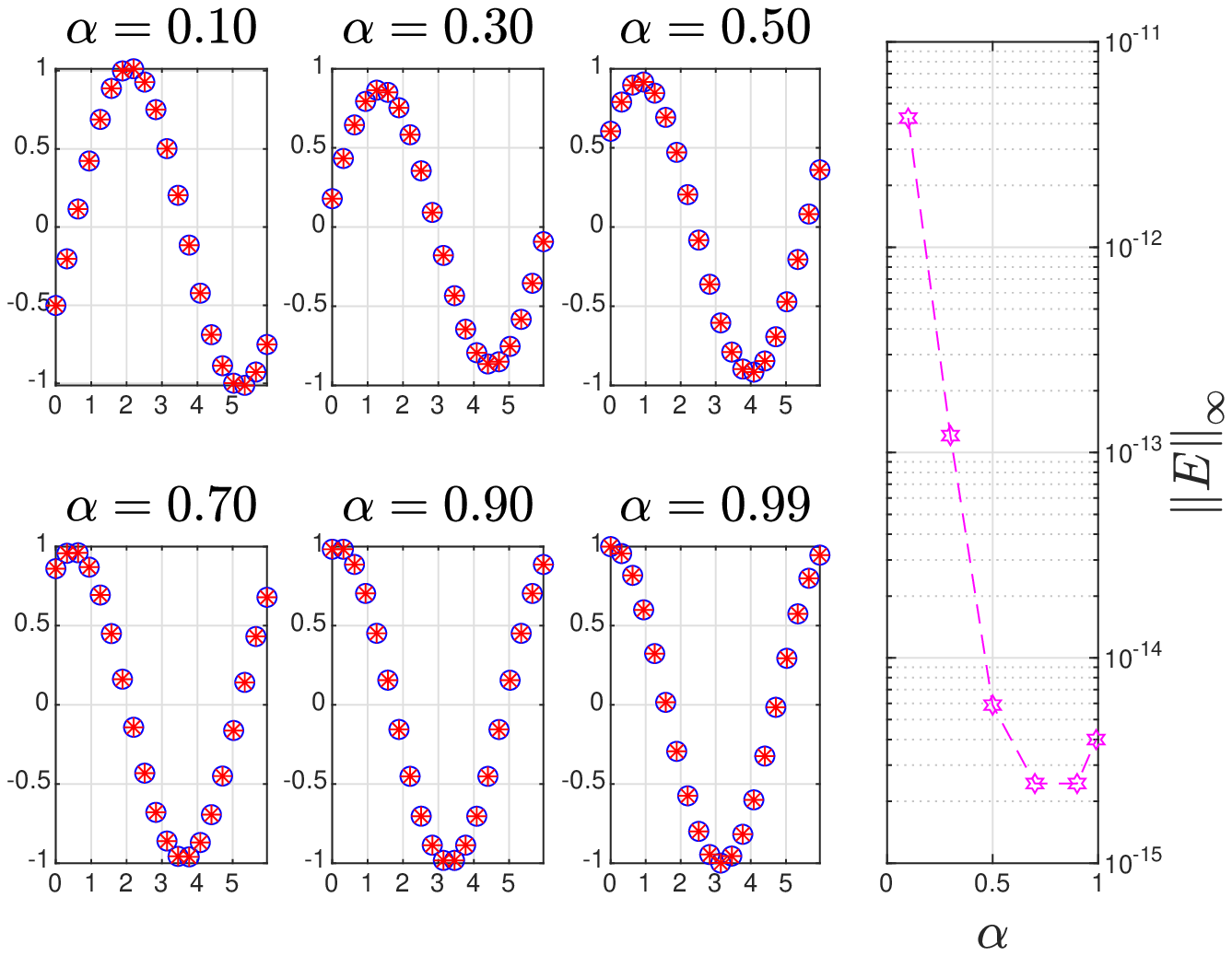}
\caption{The exact and FGPS approximate values of $\ED{L}{t}{\alpha}{\sin(t)}$ (Columns 1--3) and their corresponding maximum absolute errors (the 4th Column) for $\alpha = 0.1:0.2:0.9, 0.99$. The FGPS approximations were obtained using $N = 20, L = 30,  N_G = 1000$, and $\lambda = 0$. The exact values are shown in red color with * marker symbol, while the FGPS approximations are shown in blue colors with o marker symbol.}
\label{fig:1}
\end{figure}

\begin{figure}[t]
\centering
\includegraphics[scale=0.5]{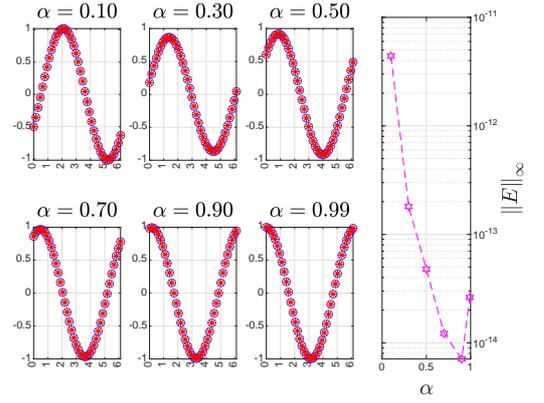}
\caption{The exact and FGPS approximate values of $\ED{L}{t}{\alpha}{\sin(t)}$ (Columns 1--3) and their corresponding maximum absolute errors (the 4th Column in log-lin scale) for $\alpha = 0.1:0.2:0.9, 0.99$. The FGPS approximations were obtained using $N = 40, L = 30,  N_G = 1000$, and $\lambda = 0$. The exact values are shown in red color with * marker symbol, while the FGPS approximations are shown in blue colors with o marker symbol.}
\label{fig:2}
\end{figure}

\section{Error and Convergence Analysis}
\label{sec:ECA1}
The following theorem states the truncation error form of the FGPSQ \eqref{eq:Quaderr1} and can be readily derived from \cite[Theorem 4.1]{ Elgindy20161}.
\begin{thm}\label{sec:erranalysgp1}
Suppose that ${\C{F}'_j}\left(t-L\,y^{\frac{1}{1-\alpha}}\right)\,\forall j \in \MBJ_N$ is approximated by the SGG interpolant obtained through interpolation at the SGG points set $\left\{\hz_{N_G,0:N_G}^{\lambda}\right\}\,\foralls N_G \in \MBZP, t \in \FOmega_T, L \in \MBRP$; cf. \cite{Elgindy20161,Elgindy2023a}. Then $\exists\,\{\zeta_{0:N-1}\} \subset (0,1)$ such that the error in the FGPSQ \eqref{eq:Quaderr1}, denoted by $E_{N,N_G,j}^{\lambda,\alpha}$, is given by
\begin{equation}
E_{N,N_G,j}^{\lambda,\alpha}(\zeta_j; t) = \frac{\psi_{L,N_G,j}^{\alpha}(\zeta_j; t)}{{({N_G} + 1)!{\mkern 1mu} K_{{N_G} + 1}^{\left( {\lambda} \right)}}} \C{I}_1^{(y)} {\hat G_{{N_G} + 1}^{\left( {\lambda} \right)}},
\end{equation}
where\\
\scalebox{0.8}{\parbox{\linewidth}{%
\begin{equation}\label{eq:Genius1}
\psi_{L,N_G,j}^{\alpha}(y;t) = \sum\limits_{k = 0}^{{N_G} + 1} {\left( {\begin{array}{*{20}{c}}
{{N_G} + 1}\\
k
\end{array}} \right) \left(\frac{L}{\alpha-1}\right)^k {{\left( {\frac{\alpha }{{\alpha  - 1}}} \right)}^{({N_G} - k + 1)}}{y^{\frac{k \alpha}{1-\alpha} + k - {N_G} - 1}}\C{F}_j^{(k + 1)}\left( {t - L{y^{\frac{1}{{1 - \alpha }}}}} \right)},
\end{equation}
}}
\\and
\begin{equation}
K_{l}^{(\lambda )} = 2^{2l - 1} \frac{{\Gamma \left( {2\lambda  + 1} \right)\Gamma \left( {l + \lambda } \right)}}{{\Gamma \left( {\lambda  + 1} \right)\Gamma \left( {l + 2\lambda } \right)}}\quad \forall l \in \mathbb{Z}_0^+,
\end{equation}
is the leading coefficient of the $l$th-degree SGG polynomial with index $\lambda, {\hat G_l^{\left( {\lambda} \right)}}$.
\end{thm}
\begin{proof}
Theorem 4.1 in \cite{Elgindy20161} immediately gives
\begin{equation}
E_{N,N_G,j}^{\lambda,\alpha}(\zeta_j) = \frac{\frac{d^{N_G+1}}{dy^{N_G+1}} {\C{F}'_j}\left(t-L\,\zeta_j^{\frac{1}{1-\alpha}}\right)}{{({N_G} + 1)!{\mkern 1mu} K_{{N_G} + 1}^{\left( {\lambda} \right)}}} \C{I}_1^{(y)} {\hat G_{{N_G} + 1}^{\left( {\lambda} \right)}}.
\end{equation}
It remains now to prove that 
\begin{equation}\label{eq:proof1}
\psi_{L,N_G,j}^{\alpha}(y;t) = \frac{d^{N_G+1}}{dy^{N_G+1}} {\C{F}'_j}\left(t-L\,y^{\frac{1}{1-\alpha}}\right). 
\end{equation}
Since ${\C{F}'_j}\left(t-L\,y^{\frac{1}{1-\alpha}}\right) = \frac{\alpha-1}{L} y^{\frac{\alpha}{\alpha-1}} \frac{d}{dy} {\C{F}_j}\left(t-L\,y^{\frac{1}{1-\alpha}}\right)$ and\\ $\frac{d^n}{dy^n} y^{\frac{\alpha}{\alpha-1}} = \left(\frac{\alpha}{\alpha-1}\right)^{(n)} y^{\frac{\alpha}{\alpha-1}-n}$, the general Leibniz rule gives
\begin{gather}
\frac{d^n}{dy^n} {\C{F}'_j}\left(t-L\,y^{\frac{1}{1-\alpha}}\right) =\\ \frac{{\alpha  - 1}}{L}\sum\limits_{k = 0}^n {\left( {\begin{array}{*{20}{c}}
n\\
k
\end{array}} \right)\frac{{{d^{n - k}}}}{{d{y^{n - k}}}}} {y^{\frac{\alpha }{{\alpha  - 1}}}}\frac{{{d^{k + 1}}}}{{d{y^{k + 1}}}}{\C{F}_j}\left( {t - L{y^{\frac{1}{{1 - \alpha }}}}} \right)\\
= \frac{{\alpha  - 1}}{L}\sum\limits_{k = 0}^n {\left( {\begin{array}{*{20}{c}}
n\\
k
\end{array}} \right)} {\left( {\frac{\alpha }{{\alpha  - 1}}} \right)^{(n - k)}}{y^{\frac{\alpha }{{\alpha  - 1}} + k - n}}\frac{{{d^{k + 1}}}}{{d{y^{k + 1}}}}{\C{F}_j}\left( {t - L{y^{\frac{1}{{1 - \alpha }}}}} \right)\\
= \sum\limits_{k = 0}^n {\left( {\begin{array}{*{20}{c}}
n\\
k
\end{array}} \right)} \left(\frac{L}{\alpha-1}\right)^k {\left( {\frac{\alpha }{{\alpha  - 1}}} \right)^{(n - k)}}{y^{\frac{k \alpha}{1-\alpha} + k - n}}\C{F}_j^{(k + 1)}\left( {t - L{y^{\frac{1}{{1 - \alpha }}}}} \right),
\end{gather}
from which Eq. \eqref{eq:proof1} follows. 
\end{proof}
The following theorem underlines the truncation error bound of the FGPSQ \eqref{eq:Quaderr1}.
\begin{thm}\label{thm:Jan212022}
Suppose that 
\begin{equation}\label{eq:tss1}
L > 1-\alpha,
\end{equation}
and the assumption of Theorem \ref{sec:erranalysgp1} holds true. Then there exist some constants ${D^{\lambda}} > 0$ and $B_1^{\lambda} > 1$, which depend only on $\lambda$, such that $\forall j \in \MBJ_N$,\\ 
\scalebox{0.9}{\parbox{\linewidth}{%
\begin{align}
	&{\left| E_{N,N_G,j}^{\lambda,\alpha}(\zeta_j; t)\right|} \le  B_{2,N,N_G}^{L,\alpha,\zeta_j}\,{2^{ - 2{N_G} - 1}}{{{e}}^{{N_G}}}{N_G}^{\lambda - {N_G} - \frac{3}{2}} \times \nonumber\\
	&{\left\{ \begin{array}{l}
	1,\quad {N_G} \ge 0 \wedge \lambda \ge 0,\\
	\displaystyle{\frac{{\Gamma \left( {\frac{{{N_G}}}{2} + 1} \right)\Gamma \left( {\lambda + \frac{1}{2}} \right)}}{{\sqrt \pi\,\Gamma \left( {\frac{{{N_G}}}{2} + \lambda + 1} \right)}}},\quad N_G \in \MBZOP \wedge  - \frac{1}{2} < \lambda < 0,\\
	\displaystyle{\frac{{2\Gamma \left( {\frac{{{N_G} + 3}}{2}} \right)\Gamma \left( {\lambda + \frac{1}{2}} \right)}}{{\sqrt \pi  \sqrt {\left( {{N_G} + 1} \right)\left( {{N_G} + 2\lambda + 1} \right)}\,\Gamma \left( {\frac{{{N_G} + 1}}{2} + \lambda} \right)}}},\quad N_G \in \MBZzereP \wedge  - \frac{1}{2} < \lambda < 0,\\
	B_1^{\lambda} {\left( {{N_G} + 1} \right)^{ - \lambda}},\quad {N_G} \to \infty  \wedge  - \frac{1}{2} < \lambda < 0,
	\end{array} \right.}\label{eq:wow1}
\end{align}
}}\vspace{-4mm}
where $B_{2,N,N_G}^{L,\alpha,\zeta_j} = {D^{\lambda}} {A_{N,N_G}^{L,\alpha,\zeta_j}}$,\\ 
\scalebox{0.9}{\parbox{\linewidth}{%
\begin{align}
&{A_{N,N_G}^{L,\alpha,\zeta_j}} = N^{N_G+1} \zeta_j^{-(N_G+1)} \left(\frac{L}{1-\alpha}\right)^{N_G+1} \gamma_{N_G}^{\alpha} \times\\
&\left\{ \begin{array}{l}
c_1 \left[ {\left( {\begin{array}{*{20}{c}}
{{N_G} + 1}\\
{{N_G}/2}
\end{array}} \right){\delta _{\frac{{{N_G}}}{2},\left\lfloor {\frac{{{N_G}}}{2}} \right\rfloor }} + \left( {\begin{array}{*{20}{c}}
{{N_G} + 1}\\
{\left\lfloor {{N_G}/2} \right\rfloor }
\end{array}} \right)\left( {1 - {\delta _{\frac{{{N_G}}}{2},\left\lfloor {\frac{{{N_G}}}{2}} \right\rfloor }}} \right)} \right],\quad N_G \in \MBZP,\\
c_2 \frac{(2 e)^{N_G/2}}{\sqrt{N_G}},\quad N_G \in\,\oset{\rightarrow}{\MBZeP},\\
c_3 \frac{(e N_G)^{\left\lfloor {N_G/2} \right\rfloor}}{{\left( {\left\lfloor {N_G/2} \right\rfloor } \right)^{\left\lfloor {N_G/2} \right\rfloor  + 1/2}}},\quad N_G \in \,\oset{\rightarrow}{\MBZOP},
\end{array} \right.\label{eq:Remar1}
\end{align}
}}\\\vspace{-4mm}

\noindent $\foralls \{c_{1:3}\} \subset \MBRP$, and $\gamma_{N_G}^{\alpha} = \sum\nolimits_{k \in {\MBJ_{{N_G} + 2}}} {\left| {{{\left( {\frac{\alpha }{{1 - \alpha }}} \right)}^{({N_G} - k + 1)}}} \right|}$.
\end{thm}
\begin{proof}
First notice that ${\C{F}_j} \in C^{\infty}(\MBR)\,\forall j \in \MBJ_N$. Through Eq. \eqref{eq:CFL2} and mathematical induction, we can easily show that the $n$th-derivative of ${\C{F}'_j}$ is given by\\
\scalebox{0.9}{\parbox{\linewidth}{%
\begin{equation}\label{eq:CFL2nth1}
\C{F}_j^{(n+1)}(t) = (-1)^{\left\lfloor {\frac{{n + 2}}{2}} \right\rfloor} \frac{1}{N} \left(\frac{2\pi}{T}\right)^{n+1} \sumd\sum\limits_{\left| k \right| \le N/2} {k^{n+1} \sin\left({\omega _k}(t - {t_{N,j}}) + {\delta _{\frac{{n + 1}}{2},\left\lfloor {\frac{{n + 1}}{2}} \right\rfloor }}\frac{\pi }{2}\right)}.
\end{equation}  
}}\\
Therefore, 
\begin{equation}\label{eq:LFOrder1}
\C{F}_j^{(n+1)}(t) = O\left(N^n\right),\quad\text{ as }N \to \infty\,\forall t \in \FOmega_T, n \in \MBZP.
\end{equation}
Notice also that 
\[\mathop {\max }\limits_{k \in {\MBJ_{n + 2}}} \left( {\begin{array}{*{20}{c}}
{n + 1}\\
k
\end{array}} \right) = \left\{ \begin{array}{l}
\left( {\begin{array}{*{20}{c}}
{n + 1}\\
{n/2}
\end{array}} \right),\quad n \in \MBZeP,\\
\left( {\begin{array}{*{20}{c}}
{n + 1}\\
{\left\lfloor {n/2} \right\rfloor }
\end{array}} \right),\quad n \in \MBZOP.
\end{array} \right.\]
Now, since $\left( {\begin{array}{*{20}{c}}
n\\
k
\end{array}} \right) \le n^k/k!\;\forall 1 \le k \le n$, and $(n+1)! \approx \sqrt{2 \pi} n^{3/2} (n/e)^n$, as $n \to \infty$ \cite[Lemma 4.2]{Elgindy20161}, then $\exists\,\beta_1 \in \MBRP:$
\begin{equation}
\left( {\begin{array}{*{20}{c}}
{n + 1}\\
{n/2}
\end{array}} \right) \le \frac{{{{(n + 1)}^{n/2}}}}{{{\frac{n}{2}!}}} \le {\beta _1}\frac{{{{(2e)}^{n/2}}}}{{\sqrt n }}\quad \forall n \in \,\oset{\rightarrow}{\MBZeP}.
\end{equation}
On the other hand, since 
\begin{equation}
\left\lfloor {n/2} \right\rfloor ! \approx \sqrt {2\pi }\,{\left( {\left\lfloor {n/2} \right\rfloor  - 1} \right)^{3/2}}{\left( {\left\lfloor {n/2} \right\rfloor  - 1} \right)^{\left\lfloor {n/2} \right\rfloor  - 1}}{e^{1 - \left\lfloor {n/2} \right\rfloor }}\quad \forall n \in \,\oset{\rightarrow}{\MBZOP},
\end{equation}
then $\exists\, \beta_2 \in \MBRP:$
\begin{equation}
\left( {\begin{array}{*{20}{c}}
{n + 1}\\
{\left\lfloor {n/2} \right\rfloor }
\end{array}} \right) \le {\beta _2}\frac{{{{(en)}^{\left\lfloor {n/2} \right\rfloor }}}}{{{{\left( {\left\lfloor {n/2} \right\rfloor } \right)}^{\left\lfloor {n/2} \right\rfloor  + 1/2}}}}\quad \forall n \in \,\oset{\rightarrow}{\MBZOP}.
\end{equation}
The above results imply\\
\scalebox{0.85}{\parbox{\linewidth}{%
\begin{gather}
\left|\psi_{L,N_G,j}^{\alpha}(y;t)\right| \le \beta_3 N^{N_G+1} \zeta_j^{-(N_G+1)} \left(\frac{L}{1-\alpha}\right)^{N_G+1} \sum\limits_{k = 0}^{{N_G} + 1} {\left( {\begin{array}{*{20}{c}}
{{N_G} + 1}\\
k
\end{array}} \right) {\left|{\left( {\frac{\alpha }{{1-\alpha}}} \right)}^{({N_G} - k + 1)}\right|}}\\
\le \beta_3 N^{N_G+1} \zeta_j^{-(N_G+1)} \left(\frac{L}{1-\alpha}\right)^{N_G+1} \mathop {\max }\limits_{k \in {\MBJ_{N_G + 2}}} \left( {\begin{array}{*{20}{c}}
{{N_G} + 1}\\
k
\end{array}} \right)\sum\nolimits_{k \in {\MBJ_{{N_G} + 2}}} {\left| {{{\left( {\frac{\alpha }{{1 - \alpha }}} \right)}^{({N_G} - k + 1)}}} \right|}
\\
\le {A_{N,N_G}^{L,\alpha,\zeta_j}}\quad \foralls \beta_3 \in \MBZP.\label{eq:Genius1}
\end{gather}
}}\\
The rest of the proof follows by \cite[Theorem 5.5]{Elgindy2023a}.
\end{proof}

Theorem \ref{thm:Jan212022} shows that the FGPSQ \eqref{eq:Quaderr1} converges exponentially fast for sufficiently smooth periodic functions, as $N_G \to \infty$, while holding all other parameters fixed. To elaborate further, one may explain this fact by realizing that when the parameter $N_G$ exist in a particular factor of the error bound \eqref{eq:wow1}, it either exists in the base/power of that factor with the power/base being constant, respectively, except for the two error factors ${N_G}^{\lambda - {N_G} - \frac{3}{2}}$ and $\frac{(e N_G)^{\left\lfloor {N_G/2} \right\rfloor}}{{\left( {\left\lfloor {N_G/2} \right\rfloor } \right)^{\left\lfloor {N_G/2} \right\rfloor  + 1/2}}}$. The former factor decays exponentially while the latter increases with $N_G$. To determine the growth rate of the latter factor, notice that we can decompose it into the product of the two factors $\frac{{{{(2e)}^{\left\lfloor {{N_G}/2} \right\rfloor }}}}{{{{\left( {\left\lfloor {{N_G}/2} \right\rfloor } \right)}^{1/2}}}}$ and ${\left( {\frac{{{N_G}/2}}{{\left\lfloor {{N_G}/2} \right\rfloor }}} \right)^{\left\lfloor {{N_G}/2} \right\rfloor }}$ with 
\[\quad \mathop {\lim }\limits_{{N_G} \to \infty } {\left( {\frac{{{N_G}/2}}{{\left\lfloor {{N_G}/2} \right\rfloor }}} \right)^{\left\lfloor {{N_G}/2} \right\rfloor }} \in [1/2,2e].\]
Since the error factor ${N_G}^{\lambda - {N_G} - \frac{3}{2}}$ decays much faster than the growth rate of $\frac{{{{(2e)}^{\left\lfloor {{N_G}/2} \right\rfloor }}}}{{{{\left( {\left\lfloor {{N_G}/2} \right\rfloor } \right)}^{1/2}}}}$, the error bound \eqref{eq:wow1} decays with exponential rate.

There are a number of other important remarks that can be drawn from the above theorem by analyzing the coefficients ${A_{N,N_G}^{L,\alpha,\zeta_j}}$; we summarize them in the following items:
\begin{enumerate}[label=(\roman*)]
\item While selecting a relatively small value of the memory length $L$ is not recommended, as pointed out in \cite{bourafa2021periodic}, increasing its value while holding all other parameters fixed increases the quadrature error bound, as indicated by Eqs. \eqref{eq:wow1} and \eqref{eq:Remar1}, thus reduces the accuracy of the FGPS approximation to the periodic FD operator $\ED{L}{t}{\alpha}{}$. This theoretical remark is supported numerically by comparing the FGPS errors obtained in Figures \ref{fig:0}, \ref{fig:1}, and \ref{fig:3} using $L = 30$ and $L = 100$, respectively. The memory length selection for $\ED{L}{t}{\alpha}{}$ therefore entails a tradeoff in the sense that higher $L$ values increase the capacity of $\ED{L}{t}{\alpha}{}$ to closely simulate the exact FD of periodic functions, especially when $\alpha \to 1$, but negatively reduce the accuracy of the FGPS approximation and vice versa.
\item The larger the number of Fourier interpolant modes $N$, the larger are the expected FGPS errors! This is a striking result because the Fourier interpolants of analytic periodic functions converge exponentially as the number of Fourier interpolation modes increases, as indicated by \cite[Corollary 5.1]{Elgindy2023b}.
This observation is supported numerically by comparing the FGPS errors obtained in Figures \ref{fig:0}--\ref{fig:2}, and \ref{fig:4} for $N  = 4, 20, 40$, and $100$, respectively. Fortunately, \cite[Corollary 5.1]{Elgindy2023b} asserts that Fourier interpolation for sufficiently smooth functions converges exponentially fast using relatively coarse mesh grids; therefore, one can run the FGPS method, which we present in Section \ref{sec:FGPM1}, to solve PFOCPs using a relatively small number of collocation nodes and can still resolve both the optimal solutions and the FD of the state variables with exceedingly high accuracy.
\end{enumerate}
Figures \ref{fig:0}--\ref{fig:4} show that the FGPS errors are non-uniform with respect to $\alpha$. To analyze this behavior, we refer to the following theorem which assumes a stronger condition than Condition \eqref{eq:tss1}.

\begin{thm}\label{thm:hihi2}
Suppose that 
\begin{equation}\label{eq:lmk1}
L > e^{1-\gamma_{em}},
\end{equation}
and the assumption of Theorem \ref{sec:erranalysgp1} holds true. Then the FGPS upper error bounds are larger on the first half $\alpha$-interval than their values on the second half, except when $\alpha \to 1$, and the FGPS error bounds attain their minimum values exactly at $\alpha = 1/2$ when $N_G \to \infty$.
\end{thm}
\begin{proof}
Notice by Assumption \eqref{eq:lmk1} that
\begin{gather}
L > e^{1-\gamma_{em}} \ge e^{\digamma(2-\alpha)} > 1-\alpha\quad \forall \alpha \in \FOmega_1,
\end{gather}
therefore, Assumption \eqref{eq:tss1} holds true. Notice also that $\mu(\alpha) = L^{1-\alpha}/\Gamma(2-\alpha)$ is strictly decreasing on $\FOmega_1$, since 
\[\frac{{d\mu }}{{d\alpha }} = \frac{{{L^{1 - \alpha }}}}{{\Gamma (2 - \alpha )}}\left[ {\digamma\left( {2 - \alpha } \right) - \ln L} \right] < 0\quad \forall \alpha \in \FOmega_1.\]
On the other hand, Figure \ref{fig:5} shows several portions of the $\gamma_{N_G}^{\alpha}$ graph on the $\alpha$-interval $(0, 1)$, where we note that, for $N_G \in \,\oset{\rightarrow}{\MBZP}$, the size of $\gamma_{N_G}^{\alpha}$ starts off large near $\alpha = 0$ and gradually increases for slightly increasing values of $\alpha$ before it gradually decreases as $\alpha$ moves towards the midpoint of the interval. At $\alpha = 1/2$, the ratio $\alpha/(1-\alpha) = 1$ and $\gamma_{N_G}^{\alpha}$ drops sharply and attains a minimum value $2$. The graph then shoots up wildly as we step off $\alpha = 1/2$ before it oscillates back and forth until $\alpha$ approaches the right boundary of the interval where the graph abruptly blows up. Combining these results with the fact that the size of $\gamma_{N_G}^{\alpha}$ is relatively larger on the first half of the $\alpha$ interval than its values in the second half, except in the small vicinity of the right boundary, we conclude that the FGPS errors are expected to be larger on the first half $\alpha$-interval than their values on the second half, except when $\alpha \to 1$, and the FGPS errors attain their minimum values exactly at $\alpha = 1/2$ when $N_G \to \infty$.
\end{proof}
Figures \ref{fig:6} and \ref{fig:10} emphasize Remarks (i) and (ii) and Theorem \ref{thm:hihi2} by displaying the maximum absolute error surfaces of the FGPS approximations for a spectrum of $(L,\alpha)$ values. The figures clearly show a decline in the precision of the FGPS approximations when $L$ or $N$ increases, $\alpha \to 1$, or $\alpha$ gradually moves away from zero to a certain limit in the first half of $\FOmega_1$. On the other hand, FGPS approximations generally improve as $\alpha \to 1/2$.

It is important to observe that this error behavior of FGPS approximations is often expected whenever the assumptions of Theorem \ref{thm:hihi2} hold true. However, the error growth may be significantly altered if any of these conditions are violated. Indeed, Figure \ref{fig:12} shows one possible scenario for what would happen when condition \eqref{eq:lmk1} is dropped. In particular, for $L = 1.52$ and $N = 100$, the FGPS errors generally grow monotonically with increasing values of $\alpha \in (0, 1)$ in contrast with the standard error pattern observed in the previous numerical simulations. 

\begin{figure}[t]
\centering
\includegraphics[scale=0.5]{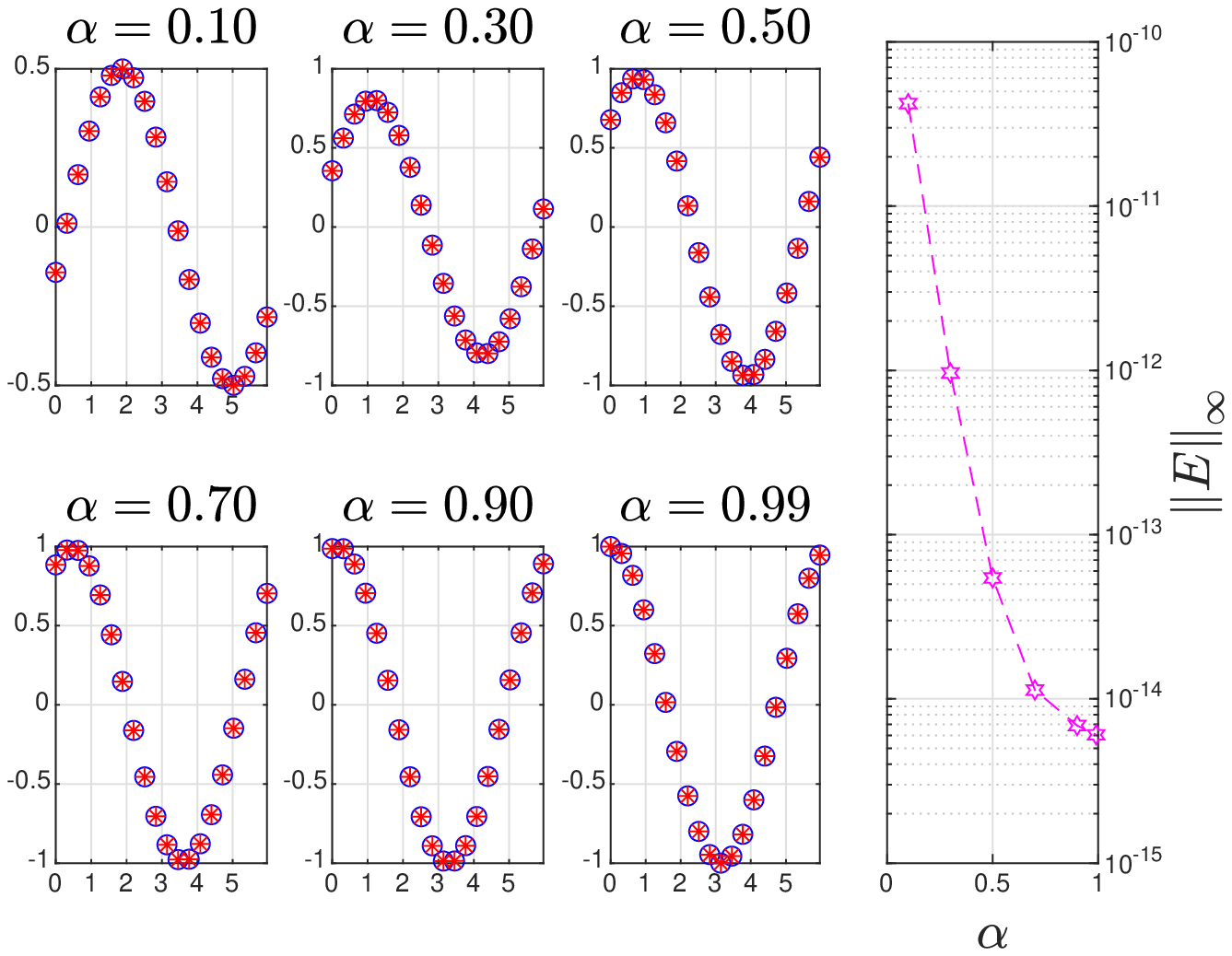}
\caption{The exact and FGPS approximate values of $\ED{L}{t}{\alpha}{\sin(t)}$ (Columns 1--3) and their corresponding maximum absolute errors (the 4th Column in log-lin scale) for $\alpha = 0.1:0.2:0.9, 0.99$. The FGPS approximations were obtained using $N = 20, L = 100,  N_G = 1000$, and $\lambda = 0$. The exact values are shown in red color with * marker symbol, while the FGPS approximations are shown in blue colors with o marker symbol.}
\label{fig:3}
\end{figure}

\begin{figure}[t]
\centering
\includegraphics[scale=0.5]{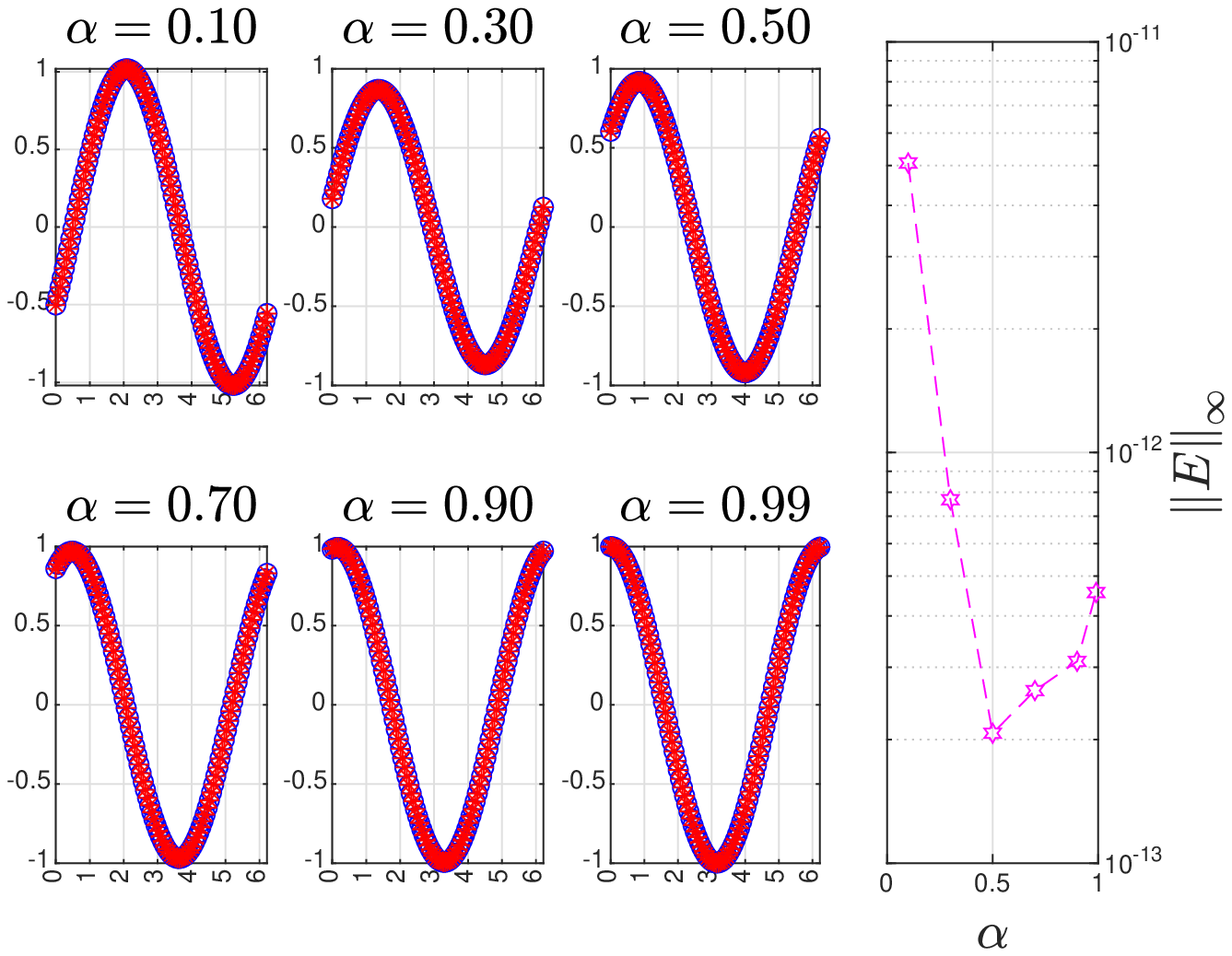}
\caption{The exact and FGPS approximate values of $\ED{L}{t}{\alpha}{\sin(t)}$ (Columns 1--3) and their corresponding maximum absolute errors (the 4th Column in log-lin scale) for $\alpha = 0.1:0.2:0.9, 0.99$. The FGPS approximations were obtained using $N = 100, L = 30,  N_G = 1000$, and $\lambda = 0$. The exact values are shown in red color with * marker symbol, while the FGPS approximations are shown in blue colors with o marker symbol.}
\label{fig:4}
\end{figure}

\begin{figure}[t]
\centering
\includegraphics[scale=0.25]{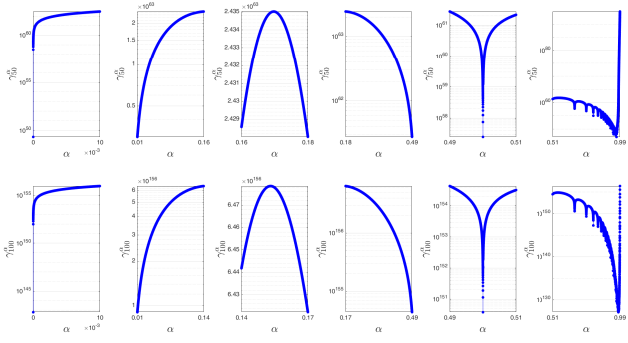}
\caption{The plots of $\gamma_{N_G}^{\alpha}$ in log-lin scale for several spectra of $\alpha$ values in $(0, 1)$. The first and second rows show the plots for $N_G = 50$ and $100$, respectively.}
\label{fig:5}
\end{figure}

\begin{figure}[t]
\centering
\includegraphics[scale=0.6]{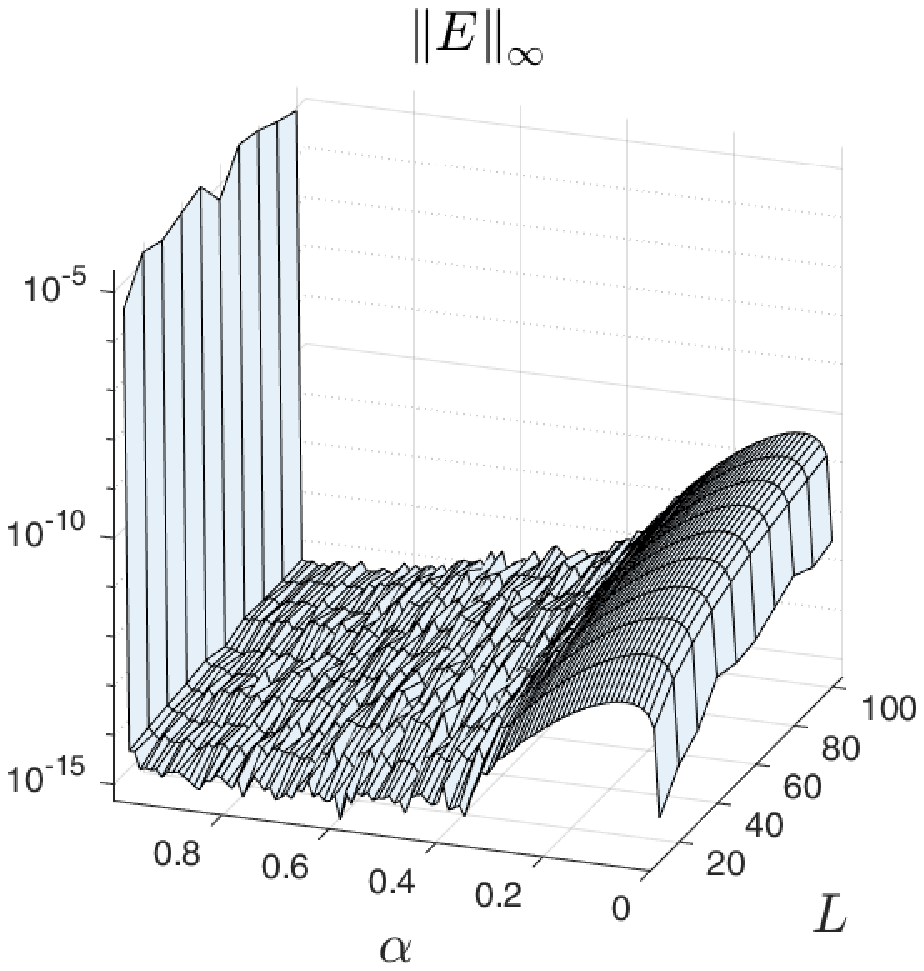}
\caption{The surface of the maximum absolute error as a function of the parameters $L$ and $\alpha$, while holding $N$ fixed at $20$. The 3D surface plot was generated above a grid in the $L\alpha$-plane defined by $L =$ 10:10:100 and a row vector of $100$ evenly spaced $\alpha$ values between $3 \epsilon_{\text{mach}}$ and $0.99999$. The axes of the $L\alpha$-plane are in lin-lin scale, while the $z$-axis is in logarithmic scale.}
\label{fig:6}
\end{figure}

\begin{figure}[t]
\centering
\includegraphics[scale=0.6]{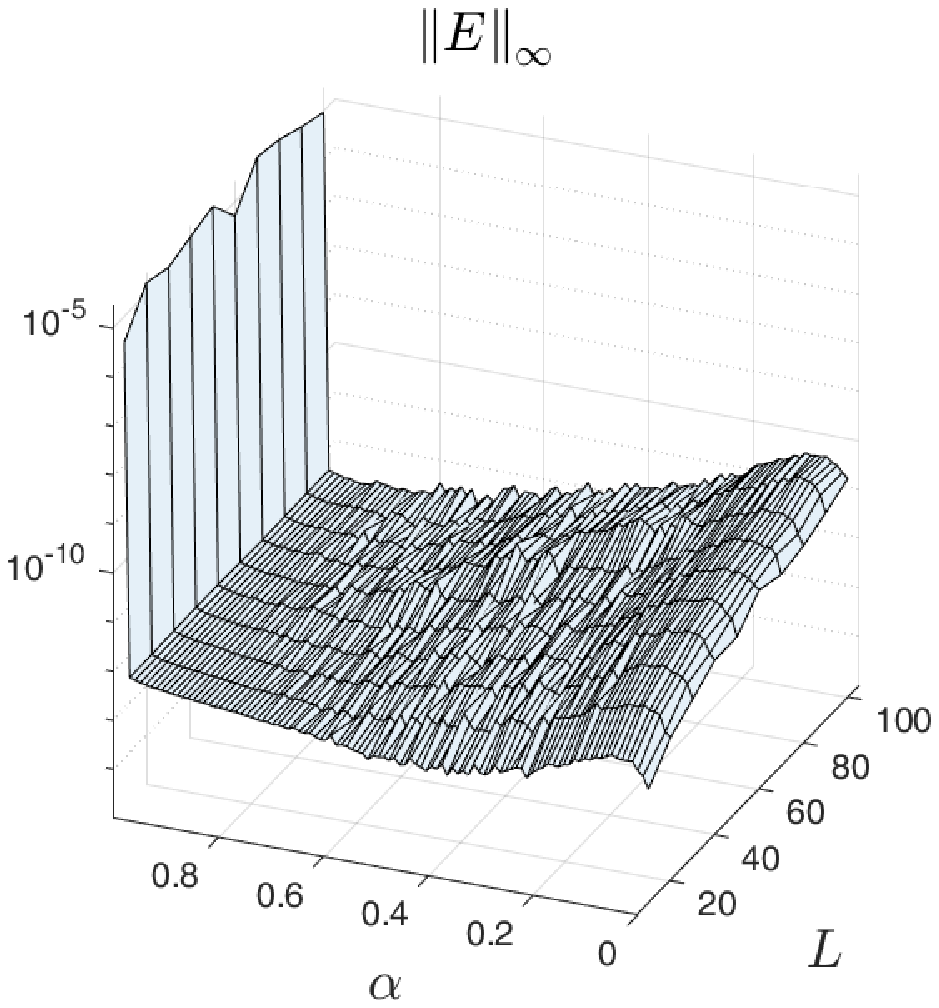}
\caption{The surface of the maximum absolute error as a function of the parameters $L$ and $\alpha$, while holding $N$ fixed at $100$. The 3D surface plot was generated above a grid in the $L\alpha$-plane defined by $L =$ 10:10:100 and a row vector of $100$ evenly spaced $\alpha$ values between $3 \epsilon_{\text{mach}}$ and $0.99999$. The axes of the $L\alpha$-plane are in lin-lin scale, while the $z$-axis is in logarithmic scale.}
\label{fig:10}
\end{figure}

\begin{figure}[t]
\centering
\includegraphics[scale=0.5]{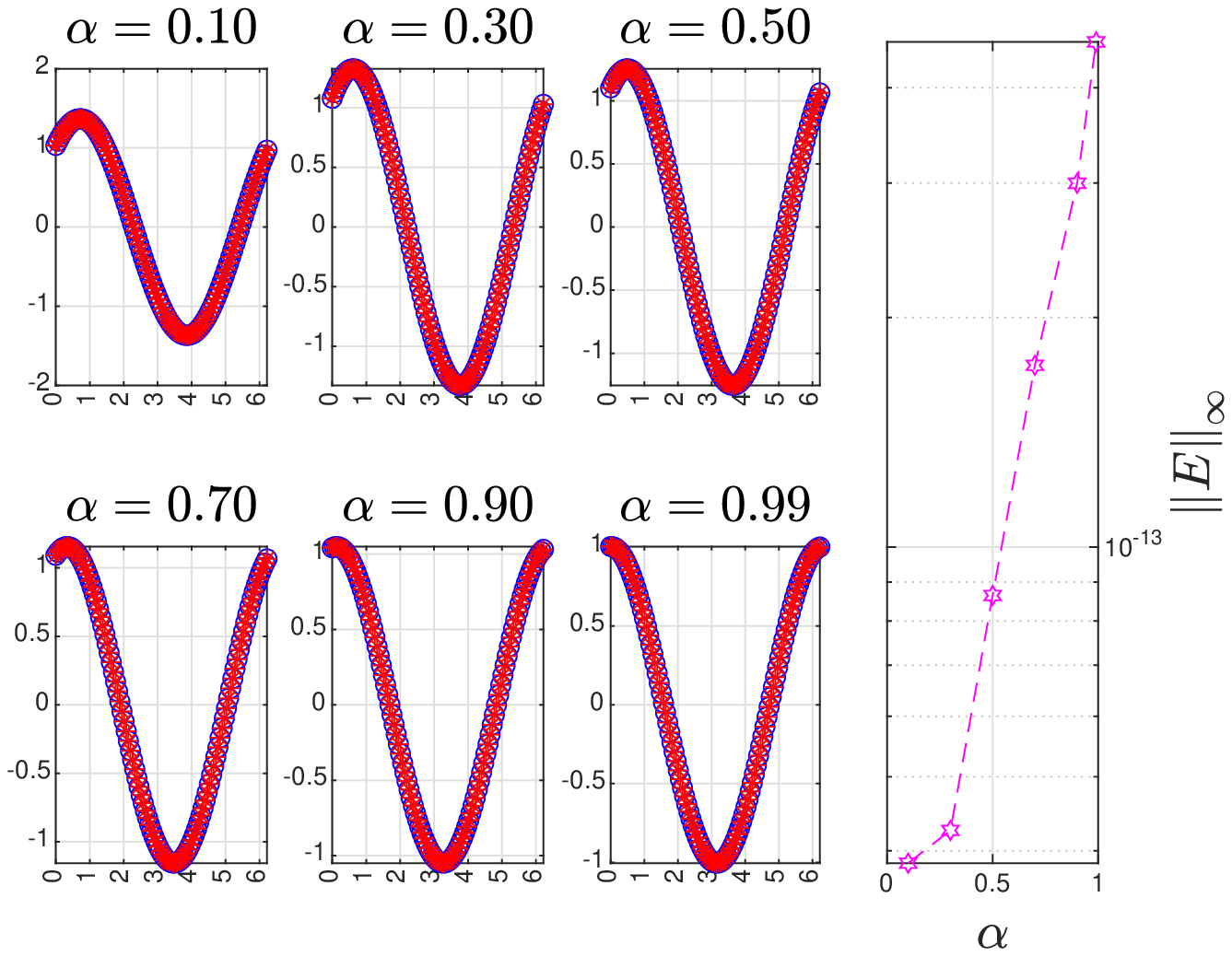}
\caption{The exact and FGPS approximate values of $\ED{L}{t}{\alpha}{\sin(t)}$ (Columns 1--3) and their corresponding maximum absolute errors (the 4th Column in log-lin scale) for $\alpha = 0.1:0.2:0.9, 0.99$. The FGPS approximations were obtained using $N = 100, L = 1.52,  N_G = 1000$, and $\lambda = 0$. The exact values are shown in red color with * marker symbol, while the FGPS approximations are shown in blue colors with o marker symbol.}
\label{fig:12}
\end{figure}

\section{The FGPS Method}
\label{sec:FGPM1}
The collocation of PFOCP \eqref{eq:OC1} in the Fourier physical space at the set of mesh points $\MBS_N^T$ with the aid of \cite[Eq. (4.9)]{Elgindy2023a}, and following the writing convention presented in Section \ref{sec:PN}, allow us to easily approximate the performance index through the Fourier quadrature using the following formula:
\begin{equation}\label{eq:Dperind1}
J \approx J_N = \frac{1}{N} \bmone_N^{\top}\,g(\bmx(\bmt_N), \bmu(\bmt_N), \bmt_N).
\end{equation}
The discrete fractional dynamical system of equations now reads
\begin{equation}\label{eq:DiscDynSysK1}
\ED{L}{\bmt_{N}}{\alpha}{\bmx(t)} \approx \bmf(\bmx(\bmt_N), \bmu(\bmt_N), \bmt_N),
\end{equation}
where 
\begin{align}
&\ED{L}{\bmt_{N}}{\alpha}{\bmx(t)} = \left[\ED{L}{\bmt_N}{\alpha}x_1(t), \ldots, \ED{L}{\bmt_N}{\alpha}x_{n_x}(t)\right]:\\
&\ED{L}{\bmt_N}{\alpha}x_j(t) \approx \frac{L^{1-\alpha}}{\Gamma(2-\alpha)} \left({}_{L}^E\bsC{Q}_{N_G}^{\alpha}\,x_j(\bmt_N)\right)\quad \forall j \in \MBN_{n_x},
\end{align}
and the discrete inequality constraints can be written as
\begin{equation}\label{eq:hhineqcons1}
\bmc(\bmx(\bmt_N),\bmu(\bmt_N),\bmt_N)) \le \F{O}_{N,p}.
\end{equation}
The PFOCP is now converted into a constrained nonlinear programming problem (NLP) in which the goal is to minimize the discrete performance index \eqref{eq:Dperind1} subject to the equality constraints \eqref{eq:DiscDynSysK1} and the inequality constraints \eqref{eq:hhineqcons1}. We solve the constrained NLP for the unknowns $\bmx(\bmt_N)$ and $\bmu(\bmt_N)$ using MATLAB fmincon solver with the interior-point algorithm, and the approximate optimal state and control variables can then be calculated at any point $t \in \FOmega_T$ through the Fourier PS expansions
\begin{equation}\label{eq:17}
\bmx(t) = \bmx\left(\bmt_N^{\top}\right) \bsC{F}_{N}(t)\quad \text{and}\quad\bmu(t) = \bmu\left(\bmt_N^{\top}\right) \bsC{F}_{N}(t),
\end{equation}
where $\bsC{F}_{N}(t) = \C{F}_{0:N-1}\cancbra{t}$. We refer to the FGPS method applied together with fmincon solver by the FGPS-fmincon method; moreover, the approximate optimal state and control variables obtained by this method are denoted by by $\tbmxs$ and $\tbmus$, respectively.

\section{Numerical Simulations}
\label{sec:NS1}
In this section, we demonstrate the performance of the proposed FGPS-fmincon method for solving the PFOCP \eqref{eq:OC1} with ${g}(\bmx(t), \bmu(t)$, $t) = 0.5 x_1^2$ $+0.25 x_2^4-0.5 x_2^2 + 0.12375 u^2$, $\bmf\left(\bmx(t),\bmu(t),t\right) = [x_2; u]$, $T = 4.431736$, and free periodic boundary conditions. The computations were carried out using MATLAB R2023a software installed on a personal laptop equipped with a 2.9 GHz AMD Ryzen 7 4800H CPU and 16 GB memory running on a 64-bit Windows 11 operating system. The fmincon solver was performed with initial guesses of all tens and was stopped whenever
\[\left\| {{\bmX^{(k + 1)}} - {\bmX^{(k)}}} \right\|_2 < {10^{ - 15}}\quad \text{or}\quad \left\|J_N^{(k+1)} - J_N^{(k)} \right\|_2 < {10^{ - 15}},\]
where $\bmX^{(k)} = [\bmx^{(k)}; \bmu^{(k)}]$ and $J_N^{(k)}$ denote the concatenated vector of approximate NLP minimizers and optimal cost function value at the $k$th iteration, respectively. We measure the quality of the approximations using the absolute discrete feasibility error (ADFE) at the collocation points; cf. \cite{Elgindy2023b}, denoted by $\bm{\C{E}}_N = (\C{E}_i)_{0 \le i \le N n_x-1}$, and given by
\begin{equation}\label{eq:DiscDynSysK2}
\boldsymbol{\C{E}}_N = \left|\ED{L}{\bmt_{N}}{\alpha}{\bmx(t)} - \bmf(\bmx(\bmt_N), \bmu(\bmt_N), \bmt_N)\right|.
\end{equation}
For $\alpha = 1$, the $\pi$-test guarantees the existence of periodic solutions for the integer OCP, which improves the performance index more than static extremal solutions \cite{Elnagar2004707,Elgindy2023b}. To the best of our knowledge, this problem has no exact solution even for $\alpha = 1$ and must be treated numerically. We acknowledge the work of \citet{evans1987solution} using the Lindstedt-Poincar\'{e} asymptotic series expansion and later the works of \citet{Elnagar2004707} and \citet{Elgindy2023b} using the FPS and FIPS methods, respectively. However, all  these early works were concerned with the integer OCP form associated with $\alpha = 1$. Figure \ref{fig:Fig9} shows the profiles of the approximate optimal state and control variables and the DFEs obtained at the collocation nodes set $\MBS_{12}^{4.431736}$ using the FGPS-fmincon method with the parameter values $N = 12, \alpha = 0.99, L = 30, N_G = 40$, and $\lambda = 0$. The approximate optimal performance index value obtained by the current method is $J_N \approx -4.18881033 \times 10^{-06}$, which is sufficiently close to the earlier value $-4.18880200 \times 10^{-06}$ obtained in \cite{Elnagar2004707} at $\alpha = 1$ with negligible DFEs approaching the machine epsilon. Figure \ref{fig:Fig10} shows the corresponding results when $\alpha = 1/2$ using $N_G = 1000$ and the same other parameter settings. The calculated approximate optimal performance index in this case is $J_N \approx 4.65023545 \times 10^{-18}$. The elapsed times required to run the FGPS-fmincon method using $\alpha = 1/2, 0.99$ and $N_G = 40$ were approximately $0.116$ s and $0.157$ s, respectively. The development stages of the approximate optimal state and control variables generated by the FGPS-fmincon method for some increasing values of $\alpha$ are shown in Figure \ref{fig:Fig11}.

\begin{figure}
\centering
\includegraphics[scale=0.24]{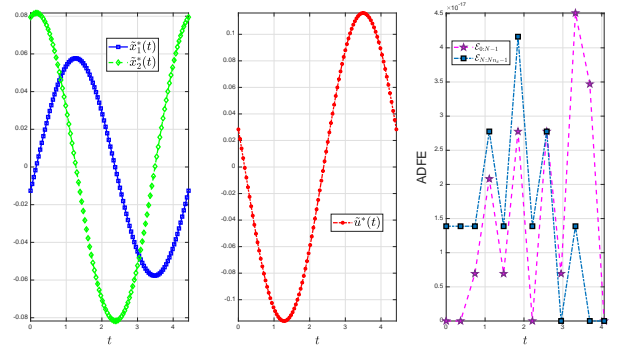}
\caption{The profiles of the approximate optimal state and control variables and the ADFEs obtained at the collocation nodes set $\MBS_{12}^{4.431736}$ using the FGPS-fmincon method with the parameter values $N = 12, \alpha = 0.99, L = 30, N_G = 40$, and $\lambda = 0$. The plots of the state and control variables were generated using $100$ linearly spaced nodes in $\FOmega_{4.431736}$.}
\label{fig:Fig9}
\end{figure}

\begin{figure}
\centering
\includegraphics[scale=0.24]{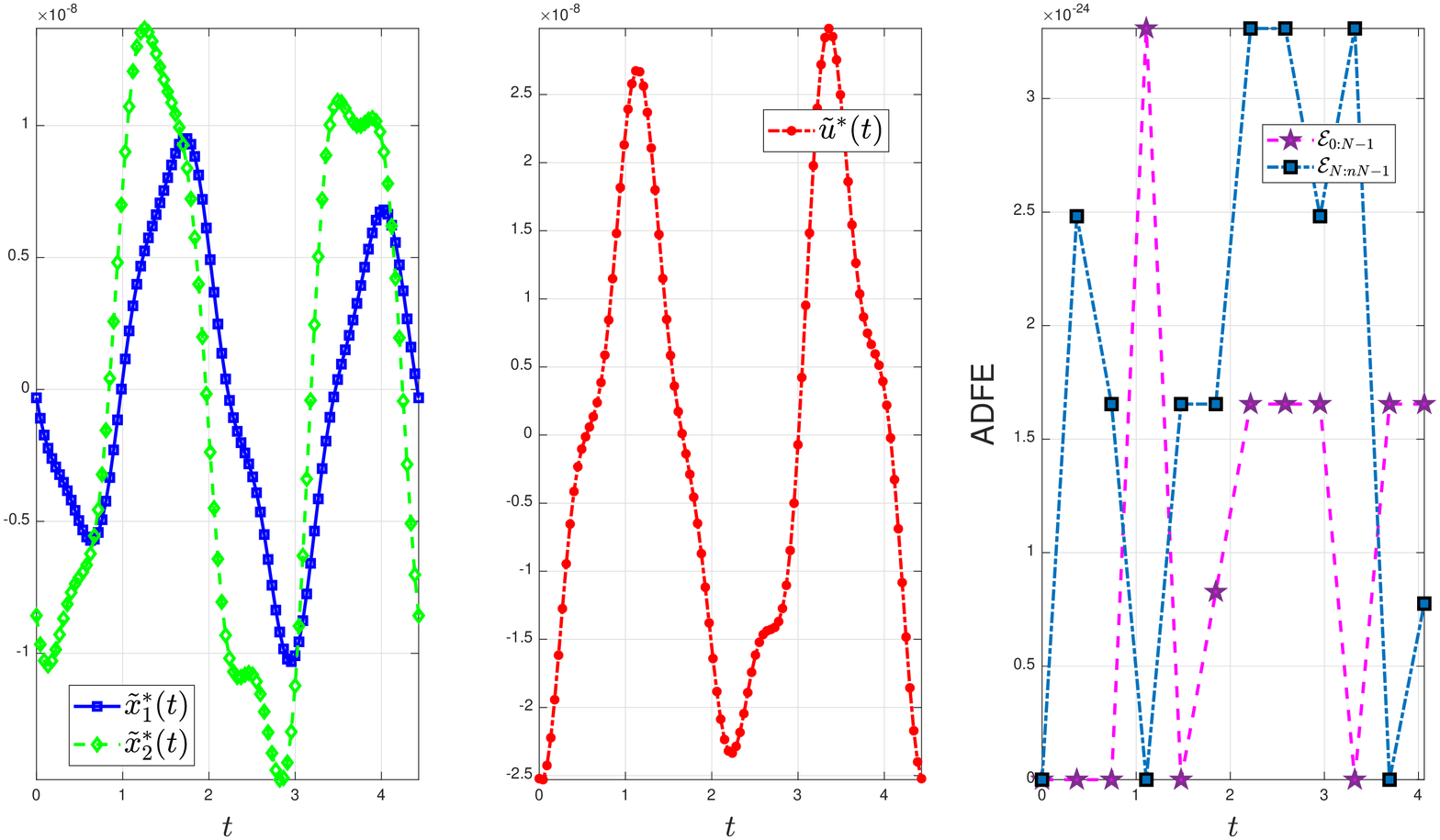}
\caption{The profiles of the approximate optimal state and control variables and the ADFEs obtained at the collocation nodes set $\MBS_{12}^{4.431736}$ using the FGPS-fmincon method with the parameter values $N = 12, \alpha = 0.5, L = 30, N_G = 1000$, and $\lambda = 0$. The plots of the state and control variables were generated using $100$ linearly spaced nodes in $\FOmega_{4.431736}$.}
\label{fig:Fig10}
\end{figure}

\begin{figure}
\centering
\includegraphics[scale=0.24]{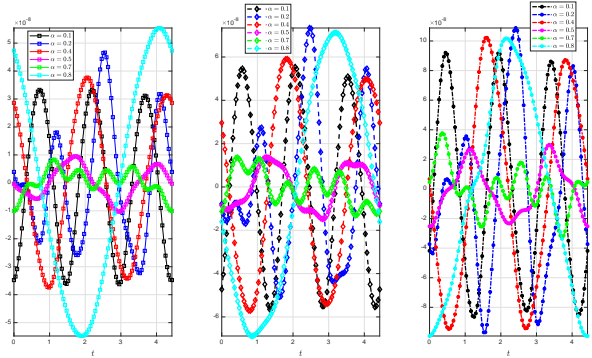}
\caption{The evolution of the approximate optimal state and control variables $\txs_1$ (left), $\txs_2$ (middle), and $\tus$ (right) for $\alpha = 0.1, 0.2, 0.4, 0.5, 0.7, 0.8$ obtained using the FGPS-fmincon method with the parameter values $N = 12, L = 30, N_G = 1000$, and $\lambda = 0$. The plots were generated using $100$ linearly spaced nodes in $\FOmega_{4.431736}$.}
\label{fig:Fig11}
\end{figure}

\section{Conclusion}
\label{sec:Conc}
This study introduced a new class of PFOCPs using periodic fractional operators. The paper also presented an accurate, efficient, and easy-to-implement numerical method for solving PFOCPs using Fourier and Gegenbauer pseudospectral methods. The proposed change of variables \eqref{eq:Elg1} largely simplifies the calculation of the periodic FD of a periodic function $f$ into the problem of evaluating the integral of the first derivative of the trigonometric Lagrange interpolating polynomial associated with the Fourier interpolation of $f$, which can be evaluated using Gegenbauer quadratures with high accuracy and efficiency. Another important contribution of this work is the introduction of the FGPSFIM ${}_{L}^E\bsC{Q}^{\alpha}_{N_G}$, which can deliver very accurate and efficient approximations to periodic FDs of periodic functions at any set of mesh points in their domains. The FGPSFIM was also shown to be a Toeplitz matrix, which can be constructed rapidly, requiring only to specify its first column and first row. A third significant contribution of this study is the derivation of Theorem \ref{thm:Jan212022}, which is extremely useful in the sense that it can predict the quality of FGPS approximations to FDs, in advance, when changing the number of Fourier interpolant modes $N$, the memory length of the periodic fractional differentiation operator $L$, the order of the fractional differentiation $\alpha$, or the size of the GG quadrature nodes grid $N_G$, while holding the other parameters fixed. The numerical results of the benchmark PFOCP demonstrated the excellent accuracy and rapid convergence of the proposed FGPS method. 

\section{Future Work}
\label{sec:FW1}
The bulk of the results in this study apply for $0 < \alpha < 1$; however, the current work can be easily extended to include PFOCPs governed by higher-order FDEs. Another future research direction includes analyzing the behavior of the error factor $\gamma_{N_G}^{\alpha}$ mathematically to provide more support to the graphical analysis of its behavior that was prescribed in the proof of Theorem \ref{thm:hihi2}. A third useful future direction is to derive the necessary conditions of optimality of the solutions of Problem \eqref{eq:OC1} using Pontryagin's Maximum Principle in order to determine its exact solution, if exists, and aid in comparing the performance of competitive numerical algorithms employed to solve it.

\section*{Declarations}
\subsection*{Competing Interests}
The author declares there is no conflict of interests.

\subsection*{Availability of Supporting Data}
The author declares that the data supporting the findings of this study are available within the article.

\subsection*{Ethical Approval and Consent to Participate and Publish}
Not Applicable.

\subsection*{Human and Animal Ethics}
Not Applicable.

\subsection*{Consent for Publication}
Not Applicable.

\subsection*{Funding}
The author received no financial support for the research, authorship, and/or publication of this article.

\subsection*{Authors' Contributions}
The author confirms sole responsibility for the following: study conception and design, data collection, analysis and interpretation of results, and manuscript preparation.

\bibliographystyle{model1-num-names}
\bibliography{Bib}
\end{document}